\documentclass[10pt]{amsart}
\usepackage{amssymb,amsbsy,amsmath,amsfonts,times, amscd}
\usepackage{latexsym,euscript,exscale}
\usepackage{graphicx,color} \usepackage{amsthm}
\usepackage{fancyhdr} \usepackage{url}
\usepackage{epigraph}\usepackage{lmodern}
\usepackage{mathrsfs}\usepackage{cancel}
\usepackage[toc,page]{appendix}\usepackage[T1]{fontenc}
\usepackage{mathtools}\usepackage[all,cmtip]{xy}
\usepackage{enumerate}\usepackage{setspace}
\usepackage{helvet}
\usepackage{float}

\numberwithin{equation}{section}
\newtheorem{theorem}{Theorem}
\newtheorem{thm}[theorem]{Theorem}

\newtheorem{cor}[theorem]{Corollary}
\newtheorem{lemma}[theorem]{Lemma}
\newtheorem{prop}[theorem]{Proposition}
\theoremstyle{definition}
\newtheorem{defi}[theorem]{Definition}
\newtheorem{claim}{Claim}
\newtheorem{problem}[theorem]{Problem}
\theoremstyle{remark}
\newtheorem{remark}[theorem]{Remark}

\newtheorem*{acknowledgments}{Acknowledgments}
\numberwithin{theorem}{section}

\newcommand {\A}{\mathbb A}

\newcommand {\M}{\mathbb M}

\newcommand{\eps}{\varepsilon}

\newcommand{\cV}{\mathcal{V}}

\newcommand{\cI}{\mathcal{I}}
\newcommand{\cE}{\mathcal{E}}
\newcommand{\cU}{\mathcal{U}}

\newcommand{\SB}{\mathrm{SB}}

\newcommand{\N}{\mathbb{N}}
\newcommand{\R}{\mathbb{R}}

\newcommand{\WF}{\mathrm{WF}}
\newcommand{\IF}{\mathrm{IF}}
\newcommand{\Tr}{\mathrm{Tr}}

\newcommand{\Lip}{\mathrm{Lip}}
\newcommand{\wlim}{\text{w-}\lim}

\newcommand{\wslim}{\text{w}^*\text{-}\lim}
\newcommand{\supp}{\mathrm{supp}}
\newcommand{\cof}{\mathrm{cof}}
\newcommand{\coKR}{\mathrm{co}\text{-}\mathrm{KR}}

\newcommand{\KR}{\mathrm{KR}}
\newcommand{\Id}{\mathrm{Id}}

\newcommand{\tn}{{\vert\kern-0.25ex\vert\kern-0.25ex\vert}}
\newcommand{\Tn}{{\big\vert\kern-0.25ex\big\vert\kern-0.25ex\big\vert}}    
\newcommand{\TN}{{\Big\vert\kern-0.25ex\Big\vert\kern-0.25ex\Big\vert}} 
\begin{document}

\title[Asymptotically uniformly smoothness and nonlinear geometry]{On asymptotically uniformly smoothness and nonlinear geometry of Banach spaces}
 
\author{B. M. Braga }
\address{Department of Mathematics and Statistics\\
 York University\\
4700 Keele St.\\ Toronto, Ontario, M3J IP3\\
Canada}\email{demendoncabraga@gmail.com}
\urladdr{https://sites.google.com/site/demendoncabraga}
\thanks{The author is supported by York Science Research Fellowship}
\date{\today}
\maketitle

\begin{abstract}
These notes concern the nonlinear geometry of Banach spaces, asymptotic uniform smoothness and several Banach-Saks-like properties. We study the existence of certain concentration inequalities in asymptotically uniformly smooth Banach spaces as well as weakly sequentially continuous coarse (Lipschitz) embeddings into those spaces. Some results concerning the descriptive set theoretical complexity of those properties are also obtained. We finish the paper with a list of open problem. 
\end{abstract}

\setcounter{tocdepth}{1}
\tableofcontents

\section{Introduction}\label{SectionIntro}

This paper mainly deals with  the existence of concentration inequalities of some specific maps into Banach spaces and asymptotic uniform smoothness. The motivation behind our approach  comes from the study of the nonlinear geometry of Banach spaces, more specifically, of the coarse and coarse Lipschitz geometry of Banach spaces and the search for properties which are stable under those kinds of embeddings (e.g., \cite{Kalton2007}, \cite{MendelNaor2008}, and \cite{BaudierLancienSchlumprecht2018}).

In  \cite{KaltonRandrianarivony2008}, N. Kalton and L. Randrianarivony proved an important concentration inequality for maps from the set of $k$-tuples of natural numbers into reflexive Banach spaces with asymptotically uniformly smooth (AUS) norms (we refer the reader to Section \ref{SectionPrelim} for any terminology not defined in  this introduction).  
Precisely, for each $k\in\N$, denote the set of all strictly increasing $k$-tuples $\bar n=(n_1,\ldots,n_k)$ of natural numbers by $[\N]^k$ and let $d_{\mathrm{H}}$ be the Hamming distance on $[\N]^k$, i.e., 
\[d_{\mathrm{H}}(\bar n,\bar m)=|\{j\in \{1,\ldots,k\}\mid n_j\neq m_j\}|,\]
for all $\bar n, \bar m\in [\N]^k$. By Theorem 4.2 of \cite{KaltonRandrianarivony2008},  if $X$ is a reflexive $p$-asymptotically uniformly smooth ($p$-AUS) Banach space, then there exists $C>0$ so that, for all $k\in\N$, all Lipschitz maps $f:([\N]^k,d_{\mathrm{H}})\to X$ and all $\eps>0$, there exists an infinite subset $\M\subset \N$ such that
\begin{equation}\label{EqKaltonRandri}
\text{diam}(f([\M]^k))\leq C\Lip(f)k^{1/p}+\eps.
\end{equation}

As noticed in \cite{LancienRaja2017}, if $X$ is not reflexive, James characterization of reflexivity (see \cite{James1964}, Theorem 1) gives a bounded sequence $(x_n)_{n}$ in $X$ so that 
$\|\sum_{i=1}^kx_{n_i}-\sum_{i=1}^kx_{m_i}\|\geq k$,
for all $n_1<\ldots< n_k<m_1<\ldots<m_k\in \N$. So, \eqref{EqKaltonRandri} does not hold for non-reflexive spaces. However, G. Lancien and M. Raja showed in \cite{LancienRaja2017} that for quasi-reflexive Banach spaces with an equivalent $p$-AUS norm, a concentration inequality still holds for interlaced tuples. Recall, a Banach space $X$ is \emph{quasi-reflexive} if $X$ has finite codimension in its bidual. It was shown in \cite{LancienRaja2017} that  if $X$ is a quasi-reflexive $p$-AUS Banach space, then there exists $C>0$ so that, for all $k\in\N$, all Lipschitz maps $f:([\N]^k,d_{\mathrm{H}})\to X$ and all $\eps>0$, there exists an infinite subset $\M\subset \N$ such that
\begin{equation}\label{EqLancienRaja}
\|f(\bar{n})-f(\bar{m})\|\leq C\Lip(f)k^{1/p}+\eps,
\end{equation}
for all interlaced $\bar n,\bar m\in [\M]^k$, i.e., for all $n_1<m_1<\ldots<n_k<m_k\in\M$ (see \cite{LancienRaja2017}, Theorem 2.2).

As a consequence, if a Banach space $X$ coarse Lispchitzly embeds into a quasi-reflexive $p$-AUS Banach space, then $X$ has the \emph{alternating $p$-Banach-Saks property}\footnote{Although this is not a standard terminology, this definition should be compared with the \emph{weak $p$-Banach-Saks property} (or simply  \emph{$p$-Banach-Saks property} in part of the literature). A Banach space $X$ has the alternating $p$-Banach-Saks property if and only if it does not contain $\ell_1$ and it has the weak $p$-Banach-Saks property (see Section \ref{SectionPrelim} for definitions). }, i.e., there exists $C>0$ such that for all sequences $(x_n)_n$ in the unit ball of $X$ and all $k\in\N$, there exists $\bar n\in [\N]^{k}$ such that $\|\sum_{j=1}^k(-1)^jx_{n_j}\|\leq Ck^{1/p}$. By investigating the relation between the duals of $X$ and its nonlinear embeddings into other spaces, we obtain a strengthening of this result. If $X$ is a Banach space, let $X^{(1)}=X^*$ and define by induction  $X^{(n+1)}=(X^{(n)})^*$ for $n\in\N$.

\begin{thm}\label{CorIteratedDualspBanSaks}
Let $p\in (1,\infty)$. Let $X$ be a Banach space and let $Y$ be a quasi-reflexive $p$-AUS Banach space. If $X$ coarse Lispchitzly embeds into $Y$, then $X^{(2\ell)}$ has the alternating $p$-Banach-Saks property for all $\ell\in\N$.  
\end{thm}

While concentration inequalities hold for maps from $[\N]^k$ into a quasi-reflexive $p$-AUS space, on the opposite side of  quasi-reflexivity, it is well known that every separable Banach space Lipschitzly embeds into $c_0$ and $c_0$ is $p$-AUS for all $p\in (1,\infty)$. In other words, $c_0$ is as asymptotically uniformly smooth  as one could possibly hope for, but still no kind of concentration inequality can hold for $c_0$. Therefore, a natural question rises: what happens between quasi-reflexivity and $c_0$? E.g.,  what can we say if (i) $X$ is complemented in $X^{**}$, (ii) $X^{**}/X$ is reflexive, or (iii) $X^{**}/X$ is AUSable? Recall, a Banach space is \emph{AUSable} (resp. $p$-\emph{AUSable}) if it has an equivalent AUS (resp. $p$-AUS) norm.

The main goal of Section \ref{SectionConcentrationInequality} and  Section \ref{SectionSeparableDuals}  is to systematically study  those so called $1/p$-concentration inequalities (i.e., inequalities as in \eqref{EqLancienRaja}) as well as  their implications. In the realm of non-reflexive Banach spaces, it is not possible to assure that the iterated duals of a Banach space are separable by simply restricting to its separable subspaces. This makes the analysis considerably more complicated since one has to work with maps from $[I_1]^k\times\ldots \times[I_\ell]^k$ into $X$, where each $I_j$ is a directed set (instead of simply $\N$). For that task,  we introduce two new properties for the class of metric spaces,  $\KR(p)$ and $\coKR(p)$. In a nutshell, $\KR(p)$ states that ``the space satisfies a $1/p$-concentration inequality'' and $\coKR(p)$ states that ``the space does not satisfy a $1/p'$-concentration inequality for any $p'>p$'', respectively (see Definition \ref{DefiKRp} and Definition \ref{DefiCoKRp}).

Vaguely speaking\footnote{We hope the reader forgive our vague statements here and we refer to Definition \ref{DefiKRp}, Definition \ref{DefiCoKRp}, Definition \ref{DefiModulusKappa} and the results quoted in (I), (II), (III) and (IV) for precise mathematical statements.}, the study of $\KR(p)$ and $\coKR(p)$ provides us with the following.

\begin{enumerate}[(I)]
\item \label{ItemHighDualsWithoutABSImpliesCoKR1} If $X^{(2\ell)}$ does not have the alternating Banach-Saks property for some $\ell\in\N$, then $X$ does not satisfy $1/p$-concentration inequalities for any $p\in (1,\infty)$ (see Corollary \ref{CorHighDualsWithoutABSImpliesCoKR1}).

\item \label{ItemXAUSCompBiDualReflAUS} Let $p\in (1,\infty)$. There exists a $p$-AUSable  Banach space $X$ which is complemented in its bidual,  $X^{**}/X$ is reflexive and AUSable, but   $X$ does not satisfy $1/p$-concentration inequalities  (see Corollary \ref{CorCorCor}\eqref{CorXAUSCompBiDualReflAUS}).

\item \label{ItemXAUSBiDualSuperAUS} Let $p\in (1,\infty)$. There exists a $p$-AUSable  Banach space $X$ so that $X^{**}/X$ is $q$-AUSable for all $q\in (1,\infty)$, but $X$  does not satisfy $1/q$-concentration inequalities for any $q\in (1,\infty)$ (see Corollary \ref{CorCorCor}\eqref{CorXAUSBiDualSuperAUS}).

\item \label{ItemSpaceWithDualsAUSComplemBidualBUTBad}Let $p\in (1,\infty)$ and $\ell\in\N$. There exists a Banach space $X$ which is complemented in its bidual, $X^{(2\ell)}$ is $p$-AUSable, but $X$ does not satisfy $1/p$-concentration inequalities (see  Corollary \ref{CorSpaceWithDualsAUSComplemBidualBUTBad}).
\end{enumerate}

Motivated by \cite{Braga2018IMRN},  in Section \ref{SectionWSCCoarseLipEmb}, we study coarse Lipschitz embeddings into $p$-AUS spaces  which are also weakly sequentially continuous. Recall, a map between two Banach spaces $f:X\to Y$ is \emph{weakly sequentially continuous} if for all  $(x_n)_n$ in $X$ which weakly converges to $x\in X$ it follows that $\wlim_nf(x_n)=f(x)$. The author  studied weakly sequentially   continiuous coarse embeddings into asymptotically uniformly convex spaces and showed  that the concept of  coarse embeddability by a map which is also weakly sequentially continuous is strictly weaker than isomorphic embeddability and strictly stronger than coarse embeddability (see \cite{Braga2018IMRN}, Corollary 1.7 and Theorem 1.8). In these notes, we prove the equivalent results for weakly sequentially continuous coarse Lipschitz embeddings.

Since $\ell_1$ coarse Lipschitzly embeds into a reflexive space (see \cite{AharoniLindenstrauss1985}, Theorem 1), by Theorem 1.4 of \cite{Braga2017JFA}, $\ell_1$ coarse Lipschitzly embeds into a reflexive space by a continuous map. As $\ell_1$ is a \emph{Schur space}, i.e., every weakly convergent sequence converges in norm, it follows that $\ell_1$ weakly sequentially continuously coarse Lipschitzly embeds into a reflexive space.  In particular, coarse Lipschitz  embeddability by weakly sequentially continuous maps is strictly weaker than isomorphic embeddability. However, since $\ell_1$ is a Schur space, this example is quite unsettling. Fortunately, the next result provide us with better examples and it gives a negative answer to Problem 5.4 of \cite{Braga2018IMRN}. Banach spaces $X$ and $Y$ are \emph{weakly sequentially homeomorphically Lipschitzly equivalent} if there exists a Lipschitz isomorphism $f:X\to Y$ such that $f$ and $f^{-1}$ are weakly sequentially continuous.

\begin{thm}\label{ThmWSCCoarseEmbWeakerIsomorphcEmb}
There exist weakly sequentially homeomorphically Lipschitzly equivalent  Banach spaces $X$ and $Y$ without the Schur property  such that $X$ does not linearly embed into $Y$. In particular,  the existence of a  weakly sequentially continuous coarse Lipschitz embedding between Banach spaces is  strictly weaker than the existence of an isomorphic embedding even for spaces without the Schur property.
\end{thm}

It has been recently proven that the class of reflexive asymptotic-$\ell_\infty$ Banach spaces is stable under coarse embeddings (see \cite{BaudierLancienMotakisSchlumprecht2018}, Theorem A). However, the following remains open.

\begin{problem}\label{ProblemCoarseLipRefleAUS}
If a Banach space $X$ coarse Lipschitzly embeds into a reflexive asymptotically uniformly smooth Banach space, does it follow that $X$ has  an asymptotically uniformly smooth renorming?
\end{problem}

Motivated by Problem \ref{ProblemCoarseLipRefleAUS}, in Section \ref{SectionWSCCoarseLipEmb}, we study several properties regarding weakly null trees in Banach spaces and obtain the following.

\begin{thm}\label{CorAUSablenessStableWSCCoarseLipEmb}
Let $p\in (1,\infty)$. Let $X$ and $Y$ be Banach spaces and assume that $X$ has  separable dual. Assume that $X$ coarse Lipschitzly embeds into $Y$ by a weakly sequentially continuous map. If $Y$ is $p$-AUSable, then $X$ is $p'$-AUSable, for all $p'\in (1,p)$. In particular, if $Y$ is AUSable, then $X$ is AUSable.
\end{thm}

Since every separable Banach space Lipschitzly embeds into $c_0$ and since $c_0$ is $p$-AUS, for all $p\in(1,\infty)$, Theorem \ref{CorAUSablenessStableWSCCoarseLipEmb} gives us many spaces which
coarse Lipschitzly embed into $c_0$, but not by a weakly sequentially continuous map. In particular, we obtain the next corollary.

\begin{cor}
The existence of a  weakly sequentially continuous coarse Lipschitz embedding between Banach spaces is  strictly stronger  than the existence of a coarse Lipschitz embedding. \qed
\end{cor}

In case of embeddings into $c_0$, our methods actually give us a much stronger result.

\begin{thm}\label{ThmWSCCoarseIntoc0}
Let $X$ be a Banach space not containing $\ell_1$. If $X$ coarsely embeds into $c_0$ by a map which is weakly sequentially continuous, then $X$ isomorphically embeds into $c_0$.
\end{thm}

The difference between $p$-asymptotic uniform smoothness and the alternating $p$-Banach-Saks property lies in the core of this paper. More precisely, we are interested in how far  a Banach space with the alternating $p$-Banach-Saks property is from being $p$-AUSable. With that in mind,  Section \ref{SectionComplexity} deals with the descriptive set theoretical aspect of those properties. Presicely, let $\SB$ denote the set of all subspaces of $C[0,1]$ endowed with the Effros-Borel structure (see Section \ref{SectionComplexity} for more details), so $\SB$ is a standard Borel space.  Theorem \ref{pBSnottpBS} provides an example of a Banach space which has the alternating $p$-Banach-Saks property but does not have an AUS renorming. This construction   gives us that the descriptive complexity of those classes in $\SB$  are different (see Theorem \ref{pBScompcoana}). 

We finish Section \ref{SectionComplexity} with a miscellaneous result which is also given by descriptive set theoretical methods. Precisely, the next result holds.

\begin{thm}\label{ThmCompBanachSaksUniversal}
Let $X$ be separable Banach space and assume that every Banach space with the Banach-Saks property  coarsely embeds into $ X$. Then every separable Banach space coarsely embeds into $X$. In particular, there exists $n\in\N$, so that $X^{(n)}$ is not separable.
\end{thm}

This paper is organized as follows. In Section \ref{SectionPrelim}, we cover the main definitions and terminologies needed for these notes. Section \ref{SectionConcentrationInequality} concerns  concentration inequalities. We introduce properties $\KR(p)$ and $\coKR(p)$, show that quasi-reflexive $p$-AUS spaces have $\KR(p)$ and prove Theorem \ref{CorIteratedDualspBanSaks} and \eqref{ItemHighDualsWithoutABSImpliesCoKR1} above.  Section \ref{SectionSeparableDuals} deals with concentration inequalities in the case where the spaces have  iterated duals of sufficiently high order separable, which gives us \eqref{ItemXAUSCompBiDualReflAUS}, \eqref{ItemXAUSBiDualSuperAUS} and \eqref{ItemSpaceWithDualsAUSComplemBidualBUTBad} above. In Section \ref{SectionWSCLipEquiv}, we prove Theorem \ref{ThmWSCCoarseEmbWeakerIsomorphcEmb}, and in Section  \ref{SectionWSCCoarseLipEmb} we deal with coarse (resp. coarse Lipschitz) embeddings which are also weakly sequentially continuous into spaces with good asymptotic smooth properties. Section \ref{SectionComplexity} concerns the complexity of some of the main properties which appear in this paper. At last,  in Section \ref{SectionProblems}, we finish this paper with a list of questions.

\section{Preliminaries}\label{SectionPrelim}

Let $(X,d)$ and $(Y,\partial)$ be metric spaces. Given a map $f:X\to Y$, define the modulus $\omega_f,\rho_f:[0,\infty)\to[0,\infty]$ by setting
\[\omega_f(t)=\sup\{\partial(f(x),f(y))\mid d(x,y) \leq t\}\]
and
\[\rho_f(t)=\inf\{\partial(f(x),f(y))\mid d(x,y) \geq t\},\]
for all $t\geq 0$. So
\[\omega_f( d(x,y))\leq \partial (f(x),f(y))\leq\rho_f(d(x,y)),\]
for all $x,y\in X$. The map $f$ is called \emph{coarse} if $\omega_f(t)<\infty$, for all $t\geq 0$, and it is called \emph{expanding} if $\lim_{t\to \infty}\rho_f(t)=\infty$. The map $f$ is called a \emph{coarse embedding} if it is both coarse and expanding. The map $f$ is \emph{cobounded} if $\sup_{y\in Y}d(y,f(X))<\infty$, and a coarse embedding which is also cobounded is  called a \emph{coarse equivalence}.

If there exists $L\geq 0$ such that $\omega_f(t)\leq Lt$, for all $t\geq 0$, $f$ is called \emph{Lipschitz}, and the \emph{Lipschitz constant of $f$}, denoted by $\Lip(f)$, is the infimum of all such $L\geq 0$. If $f$ is Lipschitz, $f^{-1}$ exists and it is Lipschitz,then  $f$ is called a \emph{Lispchitz embedding}. If there exists $L\geq 0$ such that $\omega_f(t)\leq Lt+L$ and $\rho_f(t)\geq L^{-1}t-L$, for all $t\geq 0$, $f$ is called a \emph{coarse Lipschitz embedding}.

 Given a Banach space $(X,\|\cdot\|_X)$, we view it as a metric space with metric $\|\cdot-\cdot\|_X$. Unless there is any chance of confusion, we always omit the index in the norm $\|\cdot\|_X$ and simply write $\|\cdot\|$. We denote the  closed unit ball of a Banach space $X$  by $B_X$ and its unit sphere by $\partial B_X$.

\subsection{Banach-Saks properties and asymptotic uniform smoothness}

A Banach space $X$ is said to have the \emph{Banach-Saks property} if  every bounded sequence $(x_n)_n$ in $X$ has  a subsequence $(x'_n)_n$  such that its  sequence of  Ces\`{a}ro means $(\frac{1}{n}\sum_{i=1}^nx'_i)_k$ converges.  We say that $X$ has the \emph{weak Banach-Saks property} if the same holds for every weakly null sequence. A Banach space $X$ is said to have the \emph{alternating Banach-Saks property} if  every bounded sequence $(x_n)_n$ in $X$ has a subsequence $(x'_n)_n$  such that its sequence of alternating Ces\`{a}ro means $(\frac{1}{n}\sum_{i=1}^n(-1)^ix'_i)_k$ converges.

\begin{defi}
Let $X$ be a Banach space and $p\in (1,\infty]$. For $C>0$, $X$ has the \emph{alternating  $p$-Banach-Saks property with constant $C$} (resp. \emph{weak $p$-Banach-Saks property with constant $C$}) if   for every sequence $(x_n)_n$ in $B_X$ (resp. weakly null sequence $(x_n)_n$ in $B_X$) and every $k\in\N$, there exists an infinite subset $\M\subset \N$ such that 
\[\Big\|\sum_{i=1}^k(-1)^i x_{n_i}\Big\|\leq Ck^{1/p}\text{ and }\Big(\text{resp. }\Big\|\sum_{i=1}^k x_{n_i}\Big\|\leq Ck^{1/p}\Big),\]
for all $n_1<\ldots<n_k\in\M$ (if $p=\infty$, we use the convention $1/\infty=0$).  A Banach space $X$ has the  \emph{alternating  $p$-Banach-Saks property} (resp. \emph{weak $p$-Banach-Saks property})  if it has the alternating $p$-Banach-Saks (resp. weak $p$-Banach-Saks) property with constant $C$ for some $C>0$.\footnote{Although this definition of the alternating $p$-Banach-Saks property is formally stronger than the definition given in the introduction, their equivalence follows trivially from standard Ramsey theory for colorings of $[\N]^k$ (see \cite{TodorcevicBook2010}, Theorem 1.3).}
\end{defi}

 A Banach space has the alternating $p$-Banach-Saks property if and only if it does not contain $\ell_1$ and it has the weak $p$-Banach-Saks property (cf. \cite{Braga2017Studia}, Proposition 3.1).

The set of finite codimensional subspaces of a Banach space $X$ is denoted by $\mathrm{cof}(X)$.

\begin{defi} Let $X$  be a Banach space. The \emph{modulus of asymptotic uniform smoothness of $X$} is defined as
\[\overline{\rho}_X(t)=\sup_{x\in \partial B_X}\inf_{E\in\mathrm{cof}(X)}\sup_{ y\in  \partial B_E}\|x+ty\|-1.\]
The Banach space $X$ is \emph{asymptotically uniformly smooth} (\emph{AUS} for short) if \[\lim_{t\to 0^+}\frac{\overline{\rho}_X(t)}{t}=0.\] If there exists $p\in (1,\infty)$ and $C>0$ such that \[\overline{\rho}_X(t)\leq Ct^p,\] for all $t\geq 0$,  $X$ is called \emph{$p$-asymptotically uniformly smooth} (\emph{$p$-AUS} for short). If $X$ has an equivalent norm in which $X$ is AUS (resp. $p$-AUS), $X$ is said to be \emph{AUSable} (resp. \emph{$p$-AUSable}).
\end{defi}

 Every  asymptotically uniformly smooth Banach space is $p$-asymptotically uniformly smooth for some $p\in (1,\infty)$ (see \cite{Raja2013},  Theorem 1.2).  A $p$-AUSable Banach space has the weak $p$-Banach-Saks property (see \cite{DimantGonzaloJaramillo2009}, Proposition 1.3 and Proposition 1.6). Since an AUSable Banach space does not contain $\ell_1$, it follows that a $p$-AUSable Banach space has the alternating $p$-Banach-Saks property.

\subsection{Directed sets}
 Let $I$ be a set and $\preceq$ be a partial-order on $I$. We say that $(I,\preceq)$ is a \emph{directed set} if for all $u,v\in I$ there exists $i\in I$ with $u,v\preceq i$. We always omit the order $\preceq$ and simply write $I$ for a directed set. Given $k\in \N$ and a directed set $I$, define
\[[I]^k=\{(u_1,\ldots,u_k)\in I^k\mid u_1\prec \ldots \prec u_k\},\]
$[I]^{\leq k}=\cup_{n=1}^k[I]^n$  and $[I]^{<\omega}=\cup_{k\in\N}[I]^k$. We  write elements in $[I]^k$ as $\bar{u}=(u_1,\ldots,u_k)$, and let $\min (\bar{u})=u_1$ and $\max(\bar{u})=u_k$. If $u\in I$ and $\bar{u},\bar{v}\in [I]^k$, we write $\bar{u}\preceq u$ if $\max(\bar u)\preceq u$,  and write $\bar{u}\preceq \bar{v}$ if $\max(\bar u)\preceq \min(\bar v)$. Analogously,  define $\bar{u}\prec u$, $\bar{u}\succ u$,  $\bar{u}\succeq u$ and $\bar  u\prec \bar v$.

Let $v\in I$ and $\bar{v}\in [I]^k$. For   $a\in \{v,\bar{v}\}$,   define 
\[\text{Succ}_\succeq (a)=\{u\in I\mid u\succeq a\}\ \text{ and }\  \text{Succ}_\succ (a)=\{u\in I\mid u\succ a\}.\]
 A directed set $I$ is said to have \emph{infinite tail}  if $\text{Succ}_\succeq(u)$ is infinite for all $u\in I$.\footnote{Since $I$ is directed, this is equivalent to $I$ not having a maximal element.} If  $I$ has infinite tail, the relation $\preceq$  defined above defines a \emph{directed} partial order on  $[I]^k$.\footnote{Notice that   $\bar{u}\preceq \bar{v}$ if and only if $u_j\preceq v_j$ for all $j\in\{1,\ldots,j\}$ also defines a directed partial-order on $[I]^k$ and this order would work just fine for our goals.} 
 
 Let $\bar{a}=(a_1,\ldots,a_k)$ be a $k$-tuple of non-zero real numbers. Define a distance $d_{\bar{a}}$ on $[I]^{k}$ by letting, for all $\bar{u},\bar{v}\in [I]^k$, 
\[d_{\bar{a}}(\bar{u},\bar{v})=\sum_{j\in F(u,v)}|a_j|,\]
where $F(u,v)=\{j\in \{1,\ldots,k\}\mid u_j\neq v_j\}$. Given directed sets $I_1,\ldots,I_\ell$ and $k_1,\ldots,k_\ell\in \N$, write $\bar{u}\in [I_1]^{k_1}\times\ldots\times [I_\ell]^{k_\ell}$ meaning $\bar{u}=(\bar{u}_j)_{j=1}^\ell$, where $\bar{u}_j\in[I_j]^{k_j}$, for all $j\in \{1,\ldots,\ell\}$. Let $k=\sum_{j=1}^\ell k_j$ and let $\bar{a}=(\bar{a}_j)_{j=1}^\ell$  be a $k$-tuple of non-zero real numbers, i.e., each $\bar{a}_j$ is a $k_j$-tuple of non-zero real numbers. We define a distance $d_{\bar{a}}$ on $[I_1]^{k_1}\times\ldots\times [I_\ell]^{k_\ell}$ by letting 
\[d_{\bar{a}}(\bar{u},\bar{v})=\sum_{j=1}^\ell d_{\bar{a}_j}(\bar{u}_j,\bar{v}_j),\]
for all $\bar{u},\bar{v}\in[I_1]^{k_1}\times\ldots\times [I_\ell]^{k_\ell}$. If $\bar{a}$ is the $k$-tuple with all coordinates equal to $1$,  write $d_{\mathrm{H}}=d_{\bar{a}}$. So, $d_{\mathrm{H}}$ is the Hamming distance on $[I_1]^{k_1}\times\ldots\times [I_\ell]^{k_\ell}$, i.e., $d_{\mathrm{H}}$ simply counts in how many coordinates $\bar{u}$ and $\bar{v}$ differ from each other.

 A \emph{cofinal ultrafilter} $\cV$ on a directed set $I$ is an ultrafilter which contains the cofinal filter base of $I$, i.e., $\text{Succ}_\succeq(i)\in \cV$, for all $i\in I$. By Zorn's lemma, every directed set $I$ has a cofinal ultrafilter on it.

Let  $I$ be a directed set with infinite tail,  $\cV$ be an ultrafilter on $I$,  $k\in \N$, $K$ be a compact topological space and  $f:[I]^k\to K$ be a map. We make constant use of the following abbreviation:
\[\lim_{\bar{u},\cV}f(\bar{u})\coloneqq \lim_{u_1,\cV}\ldots \lim_{u_k,\cV}f(u_1,\ldots,u_k).\]
If $k>1$ and $\cV$ is cofinal on $I$, for each $\bar{u}=(u_1,\ldots,u_{k-1})\in [I]^{k-1}$, define 
\[\cV_{\bar{u}}=\Big\{V\cap\text{Succ}_\succ(u_{k-1}) \mid V\in\cV\Big\}.\]
Since $I$ has infinite tail and $\cV$ is cofinal in $I$, $\cV_{\bar{u}}$ is an ultrafilter on $\text{Succ}_\succ(u_{k-1}) $. If $K$ is a compact topological space and $f:[I]^k\to K$ is a map, we use the following abuse of notation: 
\[\lim_{v,\cV}f(\bar{u},v)\coloneqq \lim_{v,\cV_{\bar{u}}}F(v),\]
where $F:\text{Succ}_\succ(u_{k-1}) \to K$ is given by $F(v)=f(\bar{u},v)$ for all $v\in\text{Succ}_\succ(u_{k-1}) $.

Let $I$ be a directed set and   $\cV$ be an ultrafilter on $I$. We define an ultrafilter $[\cV]^2$ on $[I]^2$ by setting, for $A\subset [I]^2$, 
\[A\in [\cV]^2\ \Leftrightarrow\ \{u\in I\mid \{v\in I\mid (u,v)\in A\}\in\cV\}\in \cV.\]
Since $\cV$ is an ultrafilter, it easily follows that $[\cV]^2$ is an ultrafilter on $[I]^2$ (see \cite{TodorcevicBook2010}, Chapter 1 for more details). If  $K$  is a compact topological space and  $(a_{u,v})_{(u,v)\in [I]^2}$ is a family in $K$, we use the abbreviation 
\[\lim_{u,v,\cV}a_{u,v}\coloneqq \lim_{(u,v),[\cV]^2}a_{u,v}.\]
If $\ell\in\N$ and $(a_{\bar{u},\bar{v}})_{\bar{u},\bar{v}\in [I]^\ell}$ is a family in $K$,  we use the  abbreviation
\[\lim_{\bar{u},\bar{v},\cV}a_{\bar{u},\bar{v}}\coloneqq \lim_{u_1,v_1,\cV}\ldots\lim_{u_\ell,v_\ell,\cV}a_{\bar{u},\bar{v}}.\]

\subsection{Orlicz sequence spaces}
A non-zero function $F:[0,\infty)\to[0,\infty)$ is an \emph{Orlicz function} if it is continuous, nondecreasing,  convex, and  satisfies $F(0)=0$. For each  Orlicz function $F:[0,\infty)\to[0,\infty)$, define a set $\ell_F\subset \R^\N$ by setting 
 \[\ell_F=\Big\{(x_n)_n\in\R^\N\mid \exists\lambda>0,\ \sum_{n=1}^\infty F\Big(\frac{|x_n|}{\lambda}\Big)<\infty\Big\}.\]
For each $\bar x=(x_n)_n\in \ell_F$, let 
\[\|\bar x\|_{F}\coloneqq\inf\Big\{\lambda>0\mid  \sum_{n=1}^\infty F\Big(\frac{|x_n|}{\lambda}\Big)\leq 1\Big\}.\]
It is well known that $\|\cdot\|_F$ is a norm on $\ell_F$ which makes $\ell_F$ into a Banach space. The space $(\ell_F,\|\cdot\|_F)$ is called the \emph{Orlicz sequence space associated to $F$}. We refer to \cite{LindenstraussTzafriri1971} for more on Orlicz sequence spaces. In these notes, for each $n\in\N$, we identify $\R^n$ with its natural copy in $\R^\N$. Therefore, if $\bar x\in \R^n$, the term $\|\bar x\|_F$ is well defined.  

\begin{remark}\label{RamarkRhop}
Given a Banach space $X$, the modulus of asymptotic smoothness of $X$, $\overline{\rho}_X$, is an Orlicz function. Hence, $\|\cdot\|_{\overline{\rho}_X}$ is well defined. Let $C>0$, $p\in (1,\infty)$  and assume that $\overline{\rho}_X(t)\leq Ct^p$, for all $t\geq 0$. It is easy to see that, for all $k\in\N$,  $\|\bar 1\|_{\overline{\rho}_X}\leq C^{1/p}k^{1/p}$, where $\bar 1\in \R^k$  is the $k$-tuple whose coordinates are $1$. 
\end{remark}

\section{Coarse Lipschitz geometry and concentration inequalities}\label{SectionConcentrationInequality}

We start this section introducing properties $\KR(p)$ and $\coKR(p)$ for metric spaces. We show that quasi-reflexive Banach spaces with equivalent $p$-AUS norms have property $\KR(p)$ (see Corollary \ref{CorLancienRajaGEN}). Moreover, we provide a method to show that Banach spaces have property $\coKR(p)$ (see Theorem \ref{ThmBadDualsImpliesCoKRp}). Those results are used  to obtain applications to the coarse Lipschitz geometry of Banach spaces (see  Theorem  \ref{CorIteratedDualspBanSaks}) and to show that if any  iterated dual of even order of a Banach space has an $\ell_1$-spreading model then the Banach space has $\coKR(1)$ (see Corollary \ref{CorHighDualsWithoutABSImpliesCoKR1}).

\subsection{Properties $\KR(p)$ and $\coKR(p)$}

As mentioned in the introduction, when dealing with non-reflexive Banach spaces $X$, instead of looking at maps from $[\N]^k$ into $X$, it is more natural to work with general directed sets and study the behavior of maps which assign  tuples in some directed set to elements in $X$. The following is a central definition in these notes.

\begin{defi}\label{DefiKRp}
Let $(X,d)$ be a metric space and $p\in (1,\infty]$. 
\begin{enumerate}[(i)]
\item Given $C\geq 1$, $\ell\in\N$ and directed sets $I_1,\ldots,I_\ell$ with infinite tail, the Banach space $X$ is said to have  \emph{$\KR(p,C,\ell,(I_i)_{i=1}^\ell)$} if   for all $k\in\N$,   all Lipschitz maps 
\[f:([I_1]^{ k}\times\ldots\times [I_\ell]^{ k},d_{\mathrm{H}})\to X,\]
and all cofinal ultrafilters $\cV_i$ on $I_i$, for $i\in\{1,\ldots,\ell\}$, it follows that  
\[\lim_{\bar{u}_1,\bar{v}_1,\cV_1}\ldots\lim_{\bar{u}_\ell,\bar{v}_\ell,\cV_\ell}d\big(f(\bar{u}),f(\bar{v})\big)\leq (C\Lip(f)+C)\ell^{1/p} k^{1/p}.\]
(if $p=\infty$, we use the convention $1/\infty=0$).
\item Given $C\geq 1$, the Banach space $X$ is said to have  \emph{$\KR(p,C)$} if, for all $\ell\in \N$ and all directed sets $I_1,\ldots,I_\ell$ with infinite tail, $X$ has  $\KR(p,C,\ell, (I_i)_{i=1}^\ell)$.
\item   The Banach space $X$ is said to have  \emph{$\KR(p)$} if there exists $C\geq 1$ for which $X$  has $\KR(p,C)$.
\end{enumerate}
\end{defi}

Property $\KR(p)$ should be thought of as a concentration inequality for maps from tuples in directed sets  into $(X,d)$.

Besides the metric $d_{\mathrm{H}}$ on $[I]^k$, a different metric which gives us appropriate  lower estimates is needed. Precisely, let $I$ be a directed set and  $k\in\N$. Define $d_\Delta\coloneqq d^k_\Delta$ as the \emph{symmetric difference metric} on $[I]^{k}$, i.e., 
\[d_\Delta(\bar{u},\bar{v})=|\bar{u}\Delta\bar{v}|,\]
for all $\bar{u},\bar{v} \in[I]^{k}$. The following is a central notion in these notes and it represents an obstacle for  property $\KR(p)$  to hold.

\begin{defi}\label{DefiCoKRp}
Let $(X,d)$ be a metric space and $p\in [1,\infty)$.
\begin{enumerate}[(i)]
\item Given $C>0$, $\ell\in\N$ and directed sets $I_1,\ldots,I_\ell$ with infinite tail, the Banach space $X$ is said to have  \emph{$\coKR(p,C,\ell,(I_i)_{i=1}^\ell)$} if   for all $k\in\N$,   there exists a  $C$-Lipschitz map 
\[f:([I_1]^{ k}\times \ldots\times [I_\ell]^{ k},d_{\mathrm{H}})\to X,\]
such that, for all $\bar u_1,\bar v_1\in [I_1]^k$ and all cofinal ultrafilters $\cV_i$ on $I_i$, for $i\in\{2,\ldots,\ell\}$, it follows that  
\[\lim_{\bar{u}_2,\bar{v}_2,\cV_2}\ldots\lim_{\bar{u}_\ell,\bar{v}_\ell,\cV_\ell} d\big(f(\bar u,\bar v)\big) \geq  C^{-1}d_\Delta(\bar u_1,\bar v_1)^{1/p}-C.\footnote{Notice that $f$ is $C$-Lipschitz with respect to  the metric $d_\mathrm{H}$, not $d_\Delta$.}\]
\item Given $C>0$, the Banach space $X$ is said to have \emph{$\coKR(p,C)$} if there exists  $\ell\in\N$ and  directed sets $I_1,\ldots,I_\ell$ with infinite tail for which $X$ has  $\coKR(p,C,\ell,(I_i)_{i=1}^\ell)$.
\item  The Banach space $X$ is said to have  \emph{$\coKR(p)$} if there exists $C>0$ for which $X$  has $\coKR(p,C)$.
\end{enumerate}
\end{defi}

A word about $\coKR(p)$ is needed. It would be desirable to replace the inequality in Definition \ref{DefiCoKRp}(i) above by a simpler inequality such as 
\[d\big(f(\bar u,\bar v)\big) \geq  C^{-1}d_\Delta(\bar u,\bar v)^{1/p}-C,\]
for all $\bar{u},\bar{v}$. However, as it will be clear in the following sections, this is not  possible by our methods. However, although weaker, this formulation of $\coKR(p)$  is enough for our applications.\footnote{In Section \ref{SectionSeparableDuals}, we define a more quantitative version of $\coKR(p)$ for maps $[\N]^{\ell k}\to X$. We refer the impatient reader to Definition \ref{DefiModulusKappa}.}

The next proposition gathers some trivial features regarding properties  $\KR(p)$ and $\coKR(p)$. If $P$ is a property of metric spaces, we say that $P$ is \emph{stable under coarse Lipschitz embeddings} if the coarse Lipschitz embeddability of a metric space $(X,d)$ into a metric space $(Y,\partial)$ with property $P$ implies that $(X,d)$ has property $P$.  As usual in mathematics, if $P$ is a property, $\neg P$ denotes the negation of $P$.

\begin{prop}\label{PropKRpCoKRpProperties}
The following holds for the class of metric spaces.
\begin{enumerate}[(i)]
\item For $q,p\in [1,\infty]$ with $q<p$, $\coKR(q)$ implies $\neg\KR(p)$.
\item For all $p\in (1,\infty]$, $\KR(p)$ is stable under coarse Lipschitz embeddings.
\item For all $p\in[1,\infty)$,  $\neg\coKR(p)$ is stable under coarse Lipschitz embeddings.
\end{enumerate}\qed
\end{prop}

\begin{remark}
Proposition \ref{PropKRpCoKRpProperties} is the reason for the affine bounds in Definition \ref{DefiKRp}  and Definition \ref{DefiCoKRp}. More precisely, when working with Banach spaces, since one can always rescale a map with range in a vector space, the affine bounds are not necessary, and one can forget the ``$\pm C$'' in the right hand side of  the inequalities in Definition \ref{DefiKRp}(i)  and Definition \ref{DefiCoKRp}(ii). However, in order to have that those properties are stable under coarse Lipschitz embeddings into \emph{metric} spaces,  the ``$\pm C$'' is necessary. 
\end{remark}

\subsection{Quasi-reflexive $p$-AUS spaces}
The following proposition was proved in \cite{LancienRaja2017}, Proposition 2.1, for weak$^*$ null sequences. Since the  same proof works for arbitrary weak$^*$ null nets, we omit its proof here. 

\begin{prop}\label{PropLancienRaja}
Let $X$ be a Banach space. For all $t\in (0,1)$, all weak$^*$ null net $(x^{**}_i)_{i\in I}$ in $B_{X^{**}}$, all ultrafilters $\cU$ on $I$, and all $x\in S_X$, 
\[\lim_{i,\cU}\|x+tx^{**}_i\|\leq 1+\bar\rho_X(t,x).\]\qed
\end{prop}

The following theorem  is the  motivation for the definition of $\KR(p)$. Also, as it will be clear to the reader familiar with the concentration inequality of \cite{LancienRaja2017}, the next result is a generalization of Theorem 2.4 of \cite{LancienRaja2017} to arbitrary directed sets instead of $[\N]^k$.

\begin{thm}\label{TheoremLancienRajaGEN}
Let $k\in \N$ and $k_1,\ldots,k_\ell\in \N$ be such that $\sum_{i=1}^\ell k_i=k$. Let $I_1,\ldots,I_\ell$ be directed sets with infinite tail and let $X$ be a quasi-reflexive  Banach space. For all  $k$-tuples $\bar{a}$ of non-zero reals,   all Lipschitz maps 
\[f:([I_1]^{k_1}\times\ldots\times [I_\ell]^{k_\ell},d_{\bar{a}})\to X^{**},\]
and all cofinal ultrafilters $\cV_i$ on $I_i$, for  $i\in\{1,\ldots,\ell\}$, we have that 
\[\lim_{\bar{u}_1,\bar{v}_1,\cV_1}\ldots \lim_{\bar{u}_\ell,\bar{v}_\ell,\cV_\ell}\|f(\bar{u})-f(\bar{v})\|\leq 2e\Lip(f)\|\bar{a}\|_{\bar\rho_X}.\]
\end{thm}

Before proving Theorem \ref{TheoremLancienRajaGEN}, let us introduce a tool which will be of great help in its proof. First notice that, since  $\bar \rho_X$ is convex, by Fekete's lemma, the limit $\theta\coloneqq  \lim_{t\to\infty}\bar \rho_X(t)/t$ exists and it is easy to see that $\theta>0$. Define a sequence of  norms $(N_k)_k$ by induction as follows. Let $N_2$ be the map on $\R^2$ given by 
\[N_2(\xi_1,\xi_2)=\left\{\begin{array}{l l}
|\xi_1|\bar\rho_X\Big(\frac{|\xi_2|}{|\xi_1|}\Big)+|\xi_1|, &\xi_1\neq 0,\\
\theta|\xi_2|, &\xi_1=0.
\end{array}\right.\]
 Suppose $k\geq 3$ and that $N_{k-1}:\R^{k-1}\to \R$ has already been defined. Define $N_k:\R^k\to \R$ by letting 
\[N_k(\xi_1,\ldots,\xi_k)=N_2(N_{k-1}(\xi_1,\ldots,\xi_{k-1}),\xi_k),\ \ \text{for all}\ \ \xi_1,\ldots,\xi_k\in\R.\]
Denote the usual norm on $\R$ by $N_1$, i.e., $N_1(\xi)=|\xi|$, for all $\xi\in\R$. 

\begin{proof}[Proof of Theorem \ref{TheoremLancienRajaGEN}]
By Lemma 4.3 of \cite{Kalton2013AsymptoticStructure},  $N_k(\bar{a})\leq e\|\bar{a}\|_{\bar\rho_X}$, for all $\bar{a}\in \R^k$. Hence,  it is enough to show that, under the conditions above, \[\lim_{\bar{u}_1,\bar{v}_1,\cV_1}\ldots \lim_{\bar{u}_\ell,\bar{v}_\ell,\cV_\ell}\|f(\bar{u})-f(\bar{v})\|\leq 2\Lip(f)N_k(\bar{a}).\]
We  prove this by induction on $k$. If $k=1$, then $\ell=1$ and $k_1=1$. So, the result follows immediately since $\|f(u)-f(v)\|\leq \Lip(f)|a_1|$, for all $u,v\in I_1$. Assume the result holds for $k-1$ and let us show it holds for $k$. Fix $k_1,\ldots,k_\ell\in \N$ with $\sum_{i=1}^\ell k_i=k$. Let $I_1,\ldots, I_\ell$ be directed sets with infinite tail, $\bar{a}=(a_1,\ldots,a_{k-1})$ be a $(k-1)$-tuple of non-zero reals,  $a_k\in \R\setminus \{0\}$ and  $f:[I_1]^{k_1}\times\ldots\times [I_\ell]^{k_\ell}\to X^{**}$ be a Lipschitz map. For each $i\in \{1,\ldots,\ell\}$, fix a cofinal ultrafilter $\cV_i$ on $I_i$.

The proof splits in two cases, (i) $k_\ell=1$ and (ii) $k_\ell>1$. Since the proof of both cases are completely analogous, we only show (ii). Assume $k_\ell>1$ and define 
\[g:[I_1]^{k_1}\times\ldots\times [I_{\ell}]^{k_{\ell}-1}\to X^{**}\] by letting 
\[g(\bar{u})=\wslim_{u_k,\cV_\ell}f(\bar{u},u_k),\]
for all $\bar{u}\in [I_1]^{k_1}\times\ldots\times [I_{\ell}]^{k_{\ell}-1}$ (notice that $g$ is well defined since $I_\ell$ has infinite tail and $\cV_\ell$ is cofinal). By weak$^*$ lower semi-continuity of the norm of $X^{**}$, it follows that $\Lip(g)\leq \Lip(f)$. For each $\bar{u},\bar{v}\in [I_1]^{k_1}\times\ldots\times [I_{\ell}]^{k_{\ell}-1}$ and each $s,t\in I_\ell$ with $t\succ \bar{u}$ and $s\succ \bar{v}$, define 
\[u_{\bar{u},\bar{v},t,s}=(f(\bar{u},t)-g(\bar{u}))-(f(\bar{v},s)-g(\bar{v})).\]
By weak$^*$ lower semi-continuity of the norm of $X^{**}$, $\|u_{\bar{u},\bar{v},t,s}\|\leq 2\Lip(f)|a_k|$. Also, it is clear that $\wslim_{t,s,\cV_\ell}u_{\bar{u},\bar{v},t,s}=0$.

Since $X$ is quasi-reflexive, write $X^{**}=X\oplus E$, where $E$ is finite dimensional. Let $P_X:X^{**}\to X$ and $P_E:X^{**}\to E$ be the projections on $X$ and $E$, respectively. Define $h=P_X\circ g$ and $e=P_E\circ g$, so   $g(\bar{u})=h(\bar{u})+e(\bar{u})$, for all $\bar{u}\in [I_1]^{k_1}\times\ldots\times [I_{\ell}]^{k_{\ell}-1}$. Since $E$ is finite dimensional and $e$ is bounded, it follows that 
\[\lim_{\bar{u}_1,\cV_1} \ldots\lim_{\bar{u}_{\ell},\cV_{\ell}} e(\bar{u}_1,\ldots,\bar{u}_\ell)\]
exists. Therefore, since 
\[\lim_{\bar{u}_1,\bar{v}_1,\cV_1}\ldots \lim_{\bar{u}_{\ell},\bar{v}_{\ell},\cV_{\ell}}e(\bar{u})=\lim_{\bar{u}_1,\cV_1} \ldots\lim_{\bar{u}_{\ell},\cV_{\ell}} e(\bar{u})=\lim_{\bar{u}_1,\bar{v}_1,\cV_1}\ldots \lim_{\bar{u}_{\ell},\bar{v}_{\ell},\cV_{\ell}}e(\bar{v}),\]
this implies that
\begin{equation}\label{EqMapE}
\lim_{\bar{u}_1,\bar{v}_1,\cV_1}\ldots \lim_{\bar{u}_{\ell},\bar{v}_{\ell},\cV_{\ell}}\|e(\bar{u})-e(\bar{v})\|=0.
\end{equation}
The induction hypothesis applied to $g$ implies that 
\[\lim_{\bar{u}_1,\bar{v}_1,\cV_1}\ldots \lim_{\bar{u}_{\ell},\bar{v}_{\ell},\cV_{\ell}}\|g(\bar{u})-g(\bar{v})\|\leq 2\Lip(f)N_{k-1}(\bar{a}).\]
Hence, by \eqref{EqMapE}, it follows that  
\begin{equation}\label{EqMapH}
\lim_{\bar{u}_1,\bar{v}_1,\cV_1}\ldots \lim_{\bar{u}_{\ell},\bar{v}_{\ell},\cV_{\ell}}\|h(\bar{u})-h(\bar{v})\|\leq 2\Lip(f)N_{k-1}(\bar{a}).
\end{equation}

Fix $\bar{u},\bar{v}\in  [I_1]^{k_1}\times\ldots\times [I_{\ell}]^{k_{\ell}-1}$. If $h(\bar{u})\neq h(\bar{v})$, Proposition \ref{PropLancienRaja} implies that
\begin{align*}
\lim_{t,s,\cV_\ell}\|h(\bar{u})-h(\bar{v})+u_{\bar{u},\bar{v},t,s}\|&\leq\|h(\bar{u})-h(\bar{v})\|\Big(1+\bar\rho_X\Big(\frac{2\Lip(f)|a_k|}{\|h(\bar{u})-h(\bar{v})\|}\Big)\Big) \\
&=N_2(\|h(\bar{u})-h(\bar{v})\|,2\Lip(f)|a_k|).
\end{align*}
If  $h(\bar{u})=h(\bar{v})$, the inequality above trivially holds. Therefore, by \eqref{EqMapH}, it follows that
\begin{align*}
\lim_{\bar{u}_1,\bar{v}_1,\cV_1}\ldots \lim_{\bar{u}_{\ell},\bar{v}_{\ell},\cV_{\ell}}\lim_{t,s,\cV_\ell}\|h(\bar{u})-h(\bar{v})+u_{\bar{u},\bar{v},t,s}\|
&\leq 2\Lip(f)N_2(N_{k-1}(\bar{a}),a_k)\\
&=2\Lip(f)N_{k}(\bar{a},a_k).
\end{align*}
Since 
\[f(\bar{u},t)-f(\bar{v},s)=h(\bar{u})-h(\bar{v})+u_{\bar{u},\bar{v},t,s}+e(\bar{u})-e(\bar{v}),\]
this finishes the proof.
\end{proof}

 The following is a trivial consequence of Remark \ref{RamarkRhop} and  Theorem \ref{TheoremLancienRajaGEN}.

\begin{cor}\label{CorLancienRajaGEN}
Let $p\in (1,\infty)$. Every quasi-reflexive $p$-AUSable Banach space  has $\KR(p)$.\qed
\end{cor}

\subsection{Properties $\KR(p)$ and $\coKR(p)$, and the duals of Banach spaces}

The rest of the section is dedicated to finding sufficient conditions for a Banach space to have $\coKR(p)$. But first, we need some definitions. Let $X$ be a Banach space, $\ell\in \N$ and $I_1,\ldots,I_\ell$ be directed sets. Given $k\in\N$ and a family $(x_{\bar{u}})_{\bar{u}\in I_1\times\ldots\times I_\ell}$ in $X$,  define a map 
\[S\coloneqq S(\ell,k,(u_{\bar{u}})_{\bar{u}\in I_1\times\ldots\times I_\ell}): ( [I_1]^k\times\ldots\times [I_\ell]^k,d_{\mathrm{H}})\to X\] 
by letting
\[
S(u_1,\ldots,u_{\ell k} ) = \sum_{j=1}^k x_{u_j,u_{k+j},u_{2k+j},\ldots, u_{(\ell-1) k+j}},
\]
for all $\bar{u} \in[I_1]^k\times\ldots\times [I_\ell]^k$. Whenever there is no chance of confusion, we omit the index set of $\bar{u}$ and simply write $S\coloneqq S(\ell,k,(x_{\bar{u}})_{\bar{u}})$. Notice that, since we consider  $ [I_1]^k\times\ldots\times [I_\ell]^k$ endowed with the metric $d_{\mathrm{H}}$,  
\[\Lip(S)\leq 2 \cdot \sup_{\bar{u} \in[I_1]^k\times\ldots\times [I_\ell]^k}\|x_{\bar{u}}\|.\]

Given a Banach space $X$,  let $I_{X^*}$ denote a system of weak$^*$ open neighborhoods  of $0\in X^*$, i.e., we do not necessarily fix a particular system but simply state that $I_{X^*}$ is a \emph{fixed} system. We make $I_{X^*}$ into a directed set by ordering it with the  reverse inclusion order, i.e., $u_1\preceq u_2$ if and only if $u_2\subset u_1$, for all $u_1,u_2\in I_{X^*}$. Notice that, for all $x^*\in X^*$, 
\[\{\{x^*\}+u\mid u\in I_{X^*}\}\]
is a system of weak$^*$ open neighborhoods of $x^*$. Therefore, by Goldstine theorem, given $x^*\in B_{X^*}$, there exists a family $(x_u)_{u\in I_{X^*}}$ in $B_X$ so that $x_u\in \{x^*\}+u$, for all $u\in I_{X^*}$. In particular, $\wslim_{u\in I_{X^*}}x_u=x^*$. If $X$ is separable, we can pick $I_{X^*}$ to be countable and order isomorphic to $(\N,\leq)$. Therefore, in this case, we make the  identification  $I_{X^*}=\N$.

\begin{lemma}\label{LemmaFamilyBidualGENERAL}
Let  $X$ be a Banach space and let $I= I_{X^{**}}$. Let $\ell\in\N$ and let $I_1,\ldots,I_\ell$ be directed sets. Let $C>0$ and let $(x^ {**}_{\bar{u}})_{\bar{u}\in I_1\times\ldots\times I_\ell}$ be a family in $C\cdot B_{X^{**}}$. There exists a  family $(x_{\bar{u}})_{\bar{u}\in I_1\times\ldots\times I_\ell\times I}$ in $C\cdot B_X$ with the following property: for all $\eps>0$, all $k\in\N$, and all $\bar{u},\bar{v}\in [I_1]^k\times\ldots\times [I_\ell]^k$, there exists $i\in I$ such that,  letting  $f=S(\ell+1,k,(x_{\bar{u}})_{\bar{u}})$ and $F=S(\ell,k,(x^{**}_{\bar{u}})_{\bar{u}})$, it follows that 
\[\|F(\bar{u})-F(\bar{v})\|-\eps\leq \|f(\bar{u},\bar{v}')-f(\bar{u},\bar{v}')\|,\]
for all  $\bar{u}',\bar{v}'\in[I]^{k}$ with $\bar{u}',\bar{v}'\succ i$.
\end{lemma}

\begin{proof}
Without loss of generality, assume $(x^ {**}_{\bar{u}})_{\bar{u}\in I_1\times\ldots\times I_\ell}$ is in the unit ball $B_{X^{**}}$. For each $\bar{u}\in I_1\times\ldots\times I_\ell$,  Goldstine theorem gives a family  $(x_{\bar{u},v})_{v\in I}$ in $B_X$ so that $\wslim_{v\in I}x_{\bar{u},v}=x^{**}_{\bar{u}}$.   Let us observe that the  family $(x_{\bar{u}})_{\bar{u}\in I_1\times\ldots\times I_\ell\times I}$ has the required property. 
 
Fix $\eps>0$ and $k\in\N$. Define  $f=S(\ell+1,k,(x_{\bar{u}})_{\bar{u}})$ and $F=S(\ell,k,(x^{**}_{\bar{u}})_{\bar{u}})$, and notice that 
\[F(\bar{u})=\wslim_{\bar{u}'\in [I]^k} f(\bar{u},\bar{u}'),\]
for all $\bar{u}\in[I_1]^k\times\ldots\times [I_\ell]^k$. Fix  $\bar{u},\bar{v}\in [I_1]^k\times\ldots\times [I_\ell]^k$. Using weak$^*$ lower semi-continuity of the norm of $X^{**}$, there exists  $i\in I$ so that 
\[
 \|F(\bar{u})-F(\bar{v})\|-\eps \leq \|f(\bar{u},\bar{u}')-f(\bar{v},\bar{v}')\|
\]
for all $\bar{u}', \bar{v}'\in [I]^k$ with $\bar{u}', \bar{v}'\succ i$. \end{proof}

\begin{lemma}\label{LemmaCORSeqInHighdual}
Let $I$ be a directed set,  $X$ be a Banach space and $\ell\in \N$.  For each $j\in\{1,\ldots, \ell\}$ write $I_j=I_{X^{(2j)}}$. Let $C>0$ and let   $(z_u)_{u\in I}$ be a  family  in $C\cdot B_{X^{(2\ell)}}$.  There exists a  family $(x_{\bar{u}})_{\bar{u}\in I\times I_{1}\times\ldots\times I_\ell}$ in $C\cdot B_X$ with the following property: for all $\eps>0$, all $k\in\N$, and all $\bar{u},\bar{v}\in [I]^k$, letting $f=S(\ell+1,k,(x_{\bar{u}})_{\bar{u}})$ and $F=S(1,k,(z_u)_{u})$,
\[\begin{array}{c}
(\exists i_1\in I_1)(\forall \bar{u}'_1,\bar{v}'_1\in [I_1]^k \text{ with }\bar{u}'_1,\bar{v}'_1\succ i_1)\\
(\exists i_2\in I_2)(\forall \bar{u}'_2,\bar{v}'_2\in [I_2]^k \text{ with }\bar{u}'_2,\bar{v}'_2\succ i_2)\\
\vdots \\
(\exists i_\ell\in I_\ell)(\forall \bar{u}'_\ell,\bar{v}'_\ell\in[I_\ell]^k\text{ with }\bar{u}'_\ell,\bar{v}'_\ell\succ i_\ell)
\end{array} \]
it holds that
\[\|F(\bar{u})-F(\bar{v})\|-\eps\leq \|f(\bar{u},\bar{v}')-f(\bar{v},\bar{v}')\|.\]
\end{lemma}

\begin{proof}
We proceed by induction on $\ell$. If $\ell=1$, this is an immediate consequence of Lemma \ref{LemmaFamilyBidualGENERAL} with $\ell=1$ in its statement. Assume the result holds for $\ell-1$. and let us show it holds for $\ell$. Let $(z_u)_{u\in I}$ be a bounded sequence in $X^{(2\ell)}$. Let  $(x^{**}_{\bar{u}})_{\bar{u}\in I\times I_1\times \ldots\times I_{\ell-1}} $ be the bounded family in $X^{**}$ given my the inductive hypothesis applied to $\ell-1$, the Banach space $X^{**}$ and the family $(z_u)_{u\in I}$.  The result follows by a straightforward application of Lemma \ref{LemmaFamilyBidualGENERAL} to the bounded family $(x^{**}_{\bar{u}})_{\bar{u}\in I\times I_1\times \ldots\times I_{\ell-1}}$. 
\end{proof}

The next result will allow us to obtain applications to coarse Lipschitz embeddings between Banach spaces. In order to simplify notation, we introduce one more piece of terminology. Let $\M\subset \N$ be infinite and $k\in\N$. Define 
\[ \cI_k(\M)=\{(\bar{n},\bar{m})\in [\M]^k\times [\M]^k\mid n_1<m_1<\ldots<n_k<m_k\}.\]

\begin{thm}\label{ThmKRpImpliesIteratedDualspBanSaks}
Let $p\in (1,\infty]$. If a Banach space $X$ has $\KR(p)$, then $X^{(2\ell)}$ has the alternating $p$-Banach-Saks property for all $\ell\in\N$.  
\end{thm}

\begin{proof}
Fix $C>0$ such that $X$ has $\KR(p,C)$. Fix $\ell\in \N$ and let $(z_n)_{n\in\N}$ be a sequence in the unit ball of $X^{(2\ell)}$. For each $s\in\{1,\ldots, \ell\}$, let $I_s=I_{X^{(2s)}}$. Fix an infinite  $\M\subset \N$  and let $(x_{\bar{u}})_{\bar{u}\in \M\times I_{1}\times\ldots\times I_\ell}$ be the family in $B_X$ given by Lemma \ref{LemmaCORSeqInHighdual} applied to $(z_n)_{n\in\M}$. Fix $k\in\N$, and let  $f=S(\ell+1,k,(x_{\bar{n}})_{\bar{n}})$ and $F=S(1,k,(z_{n})_{n\in\M})$. Endowing $[\M]^{k}\times [I_1]^k\times\ldots\times [I_\ell]^k$ with the metric $d_{\mathrm{H}}$, we have    $\Lip(f)\leq 2$.

 Let $\cU$ be a cofinal ultrafilter on $\N$ and for each $i\in\{1,\ldots,\ell\}$ let $\cV_i$ be a cofinal ultrafilter on $I_i$. Since $X$ has $\KR(p,C)$, 
it follows that
\[\lim_{\bar{n},\bar{m},\cU}\lim_{\bar{u}_1,\bar{v}_1,\cV_1}\ldots \lim_{\bar{u}_\ell,\bar{v}_\ell,\cV_\ell}\| f(\bar{n},\bar{u})- f(\bar{m},\bar{v})\|\leq (C\Lip(f)+C)(\ell+1)^{1/p}k^{1/p}.\]
In order to simplify notation, let $L=2(C\Lip(f)+C)(\ell+1)^{1/p}$.  Therefore, since $\cU$ is cofinal, there exists $(\bar{n},\bar{m})\in \cI_k(\M)$ so that 
\begin{equation}\label{EqThmBanSak}
\lim_{\bar{u}_1,\bar{v}_1,\cV_1}\ldots \lim_{\bar{u}_\ell,\bar{v}_\ell,\cV_\ell}\| f(\bar{n},\bar{u})- f(\bar{m},\bar{v})\|\leq Lk^{1/p}.
\end{equation}

We now pick $i_1\in I_1,\ldots,i_\ell\in I_\ell$ and $\bar{u}_1,\bar{v}_1\in [I_1]^k,\ldots,\bar{u}_\ell,\bar{v}_\ell\in [I_\ell]^k$ by induction as follows: let $j\in \{2,\ldots,\ell-1\}$ and assume that  $i_1\in I_1,\ldots,i_{j-1}\in I_{j-1}$ and $\bar{u}_1,\bar{v}_1\in [I_1]^k,\ldots,\bar{u}_{j-1},\bar{v}_{j-1}\in [I_{j-1}]^k$ had been chosen (the first step of the induction follows similarly). By our choice of $f$ (see  Lemma \ref{LemmaCORSeqInHighdual}), pick $i_{j}\in I_{j}$ so that
\[\begin{array}{c}
(\forall \bar{u}_{j},\bar{v}_{j}\in [I_{j}]^k \text{ with }\bar{u}_{j},\bar{v}_{j}\succ i_{j})\\
(\exists i_{j+1}\in I_{j+1})(\forall \bar{u}_{j+1},\bar{v}_{j+1}\in[I_{j+1}]^k\text{ with }\bar{u}_{j+1},\bar{v}_{j+1}\succ i_{j+1})\\
\vdots\\
(\exists i_\ell\in I_\ell)(\forall \bar{u}_\ell,\bar{v}_\ell\in[I_\ell]^k\text{ with }\bar{u}_\ell,\bar{v}_\ell\succ i_\ell)
\end{array} \]
it holds that
\[\|F(\bar{n})-F(\bar{m})\|-\frac{1}{2}\leq \|f(\bar{n},\bar{u})-f(\bar{m},\bar{v})\|.\]
Since $\cV_{j}$ is cofinal in $I_{j}$, by \eqref{EqThmBanSak}, there exists  $\bar{u}_{j},\bar{v}_{j}\in [I_{j}]^k$ with $\bar{u}_{j},\bar{v}_{j}\succ i_{j}$ so that 
\[\lim_{\bar{u}_{j+1},\bar{v}_{j+1},\cV_{j+1}}\ldots \lim_{\bar{u}_\ell,\bar{v}_\ell,\cV_\ell}\| f(\bar{n},\bar{u})-f(\bar{m},\bar{v})\|\leq Lk^{1/p}+\frac{j}{2\ell}.\]
This finishes the induction.

For now on, fix $\bar{u}=(\bar{u}_j)_{j=1}^\ell$ and $\bar{v}=(\bar{u}_j)_{j=1}^\ell$. By the construction of $\bar{u}$ and $\bar{v}$, it follows that 
\[\| f(\bar{n},\bar{u})-f(\bar{m},\bar{v})\|\leq Lk^{1/p}+\frac{1}{2}\]
and 
$\|F(\bar{n})-F(\bar{m})\|\leq \|f(\bar{n},\bar{u})-f(\bar{m},\bar{v})\| +1/2$. Therefore, 
\[\| F(\bar{n})-F(\bar{m})\|\leq Lk^{1/p}+1.\]
Since $\M$ is arbitrary, the argument above shows that, for all infinite $\M\subset \N$, there exists $(\bar{n},\bar{m})\in \cI_{k}(\M)$ so that $\| F(\bar{n})-F(\bar{m})\|\leq Lk^{1/p}+1$. By standard Ramsey theory, we can choose an infinite $\M\subset \N$ so that $\| F(\bar{n})-F(\bar{m})\|\leq Lk^{1/p}+1$, for all  $(\bar{n},\bar{m})\in \cI_{k}(\M)$. By the definition of $F$, this implies that
\[\|z_{n_1}-z_{n_2}+\ldots+z_{n_{2k-1}}-z_{n_{2k}}\|\leq L k^{1/p}+1,\]
for all $\bar{n}=(n_1,\ldots,n_{2k})\in [\M]^{2k}$. Since $L$ is independent of $k$, this gives us that $X^{(2\ell)}$ has the alternating $p$-Banach-Saks property.
\end{proof}

\begin{proof}[Proof of Theorem \ref{CorIteratedDualspBanSaks}]
This is an immediate consequence of Proposition \ref{PropKRpCoKRpProperties}(ii), Corollary \ref{CorLancienRajaGEN} and Theorem \ref{ThmKRpImpliesIteratedDualspBanSaks}.
\end{proof}

We finish this section with a method to establish whether a Banach space has $\coKR(p)$. 

\begin{thm}\label{ThmBadDualsImpliesCoKRp}
Let $p\in [1,\infty)$. Let $X$ be a Banach space and $\ell\in \N$. Assume that there exists $C>0$ such that for every $k\in\N$ there exists a sequence  $(z_n)_{n\in\N}$ in $C\cdot B_{X^{(2\ell)}}$ so that 
\[\Big\|\sum_{j=1}^k\eps_jz_{n_j}\Big\|\geq k^{1/p},\]
for all $(\eps_j)_{j=1}^k\in \{-1,1\}^k$ and all $n_1<\ldots< n_k\in \N$.\footnote{This is satisfied if $\ell_p$ linearly embeds into $X^{(2\ell)}$ or, more generally, if $X^{(2\ell)}$ has an $\ell_p$-spreading model.} Then $X$ has  $\coKR(p)$. 
\end{thm}

\begin{proof}
Let $C>0$ be as above and for each $j\in \{1,\ldots,\ell\}$ let $I_j=I_{X^{(2j)}}$. Fix $k\in\N$ and let $(z_n)_n$ be a sequence in $C\cdot B_{X^{(2\ell)}}$ as in the statement above. Let $(x_{\bar u})_{\bar u\in \N\times I_1\times\ldots \times I_\ell}$ be the family in $C\cdot B_X$  given by Lemma \ref{LemmaCORSeqInHighdual}. Set $F=S(1,k,(z_n)_n)$ and $f=S(\ell+1,k,(x_{\bar u})_{\bar u})$.

 Notice that, for all $\bar{n},\bar{m}\in[\N]^{k}$, we have that
\[F(\bar{n})-F(\bar{m})=\sum_{j=1}^{ d_{\Delta}(\bar{n},\bar{m})}\eps_jz_{s_j},\]
for some  $(\eps_j)_{j=1}^{d_{\Delta}(\bar{n},\bar{m})}\in \{-1,1\}^{ d_{\Delta}(\bar{n},\bar{m})}$ and some  $s_1<\ldots< s_{d_{\Delta}(\bar{n},\bar{m})}\in \N$. Therefore, it follows that 
\[\|F(\bar{n})-F(\bar{m})\|\geq d_{\Delta}(\bar{n},\bar{m})^{1/p},\]
for all $\bar{n},\bar{m}\in [\N]^{k}$. 

Since $(x_{\bar u})_{\bar u} $ is in $C\cdot B_X$, it follows that $\Lip(f)\leq 2C$. For each $j\in\{1,\ldots,\ell\}$, let  $\cV_j$ be a cofinal ultrafilter in $I_j$. Using the conclusion of Lemma \ref{LemmaCORSeqInHighdual} for $\eps=1$, it follows that,  
\[\lim_{\bar{u}_1,\bar{v}_1,\cV_1}\ldots\lim_{\bar{u}_\ell,\bar{v}_\ell,\cV_\ell} \|f(\bar{n},\bar{u})-f(\bar{m},\bar{v})\|\geq  d_\Delta(\bar{n},\bar{m})^{1/p}-1,\]
for all $\bar{n},\bar{m}\in[\N]^{k}$. This shows that $X$ has $\coKR(p)$.
\end{proof}

\begin{cor}\label{CorHighDualsWithoutABSImpliesCoKR1}
Let $p\in [1,\infty)$. Let $X$ be a Banach space such that $X^{(2\ell)}$ does not have the alternating Banach-Saks property for some for some $\ell\in\N$.  Then $X$ has  $\coKR(1)$.
\end{cor}

\begin{proof}
This is a trivial consequence of Theorem \ref{ThmBadDualsImpliesCoKRp} and the fact that a Banach space does not have the alternating Banach-Saks property if and only if it has an $\ell_1$-spreading model (see \cite{Beauzamy1979}, Section III, Theorem 1). 
\end{proof}

\section{Banach spaces with separable iterated duals}\label{SectionSeparableDuals}

As mentioned in Section \ref{SectionConcentrationInequality}, if $X$ is separable, $I_{X^*}$ can be chosen to be order isomorphic to $(\N,\leq)$. Therefore, the results in the previous sections can be rewritten so that the families obtained are indexed over $[\N]^k$. This not only makes the statements visually more pleasant, but also, since we have the Ramsey theory machinery for colorings of $[\N]^k$, this allow us to obtain stronger results. In this section, we study those strengthenings and apply those results to the spaces constructed by J. Lindentrauss in \cite{Lindenstrauss1971} and iterations of those spaces. 

\begin{cor}\label{CorCORHighlySep}
Let $\ell\in\N$. Let $X$ be a Banach space so that $X^{(2\ell-1)}$ is separable. Let $C>0$ and let $(z_n)_{n\in\N}$ be a  sequence in $C\cdot B_{X^{(2\ell)}}$.  There exists a  family $(x_{\bar{n}})_{\bar{n}\in[\N]^{\ell+1}}$ in $C\cdot B_X$ so that, for all $\eps>0$ and all $k\in\N$, there exists an infinite subset $\M\subset \N$ so that, letting  $f=S(\ell+1,k,(x_{\bar{n}})_{\bar{n}\in[\M]^{\ell+1}})$ and $F=S(1,k,(z_n)_{n\in\M})$, we have that  
\[\|F(\bar{n})-F(\bar{m})\|-\eps\leq \|f(\bar{n},\bar{n}')-f(\bar{m},\bar{m}')\|,\]
for all $(\bar{n},\bar{n}'),(\bar{m},\bar{m}')\in[\M]^{(\ell+1)k}$, with $\bar{n},\bar{m}\in[\M]^{k}$ and $\bar{n}',\bar{m}'\in[\M]^{\ell k}$.
\end{cor}

\begin{proof}
This is a trivial consequence of Lemma \ref{LemmaCORSeqInHighdual}. Indeed, since $X^{(2\ell-1)}$ is separable, $X^{(j)}$ is separable for all $j\in\{1,\ldots,2\ell-1\}$. Hence, without loss of generality, assume $I_{X^{(2j)}}=\N$, for all $j\in\{1,\ldots,\ell\}$. So, the output of Lemma \ref{LemmaCORSeqInHighdual} is a bounded family $(x_{\bar{n}})_{\bar{n}\in[\N]^{\ell+1}}$. The conclusion now follows straightforwardly from standard Ramsey theory and the property satisfied by  $(x_{\bar{n}})_{\bar{n}\in[\N]^{\ell+1}}$  in Lemma \ref{LemmaCORSeqInHighdual}.
\end{proof}

In \cite{GuentnerKaminker2004}, the authors introduced the compression modulus of a metric space into another. We now introduce a variant of this modulus which will give us information regarding the  compression of the family $([\N]^k)_{k=1}^\infty$ into a metric space $(X,d)$. This should be seen as a  countable version of $\coKR(p)$ which  gives us  better tools to work  with Banach spaces with separable iterated duals.

\begin{defi}\label{DefiModulusKappa}
Let $\ell\in\N$ and $(X,d)$ be metric space. Define $\kappa(X,\ell)$ as the supremum of all $\alpha\in [0,1]$ for which   there exist $L\geq 1$ such that for all $k\in\N$ there exists a map $f:[\N]^{\ell k}\to X$ so that
\[\frac{1}{L}\cdot d_\Delta(\bar{n},\bar{m})^\alpha-L\leq d\big(f(\bar{n},\bar{n}'),f(\bar{m},\bar{m}')\big)\leq 
L\cdot d_{\mathrm{H}}\big((\bar{n},\bar{n}'),(\bar{m},\bar{m}')\big) ,\]
for all $(\bar{n},\bar{n}'),(\bar{m},\bar{m}')\in [\N]^{\ell k}$ with $\bar{n},\bar{m}\in[\N]^k$ and $\bar{n}',\bar{m}'\in[\N]^{(\ell-1)k}$. Define $\kappa(X)=\sup_{\ell\in\N} \kappa(X,\ell)$.
\end{defi}

Similarly as we have with $\coKR(p)$, the modulus $\kappa(X)$ is stable under coarse Lipschitz embeddings. Precisely, we have the following trivial proposition.

\begin{prop}
Let $(X,d)$ be a metric space and $\ell\in\N$. The following hold. 
\begin{enumerate}[(i)]
\item $X$ has $\coKR(p)$, for all $p\in [1,\infty)$ with $p>\kappa(X)^{-1}$.
\item If  $(Y,\partial)$ is a metric space so that $X$ coarse Lipschitzly embeds into $Y$, then $\kappa(Y,\ell)\geq \kappa(X,\ell)$.  
\end{enumerate}\qed
\end{prop}

\begin{cor}\label{CorNonColapsingEmbGraphs}
Let $\ell\in\N$ and $p\in [1,\infty)$. Let $X$ be a Banach space so that $X^{(2\ell-1)}$ is separable and let $(z_n)_{n\in\N}$ be a bounded sequence in $X^{(2\ell)}$ so that 
\[\Big\|\sum_{j=1}^k\eps_jz_{n_j}\Big\|\geq k^{1/p},\]
for all $k\in\N$, all $(\eps_j)_{j=1}^k\in \{-1,1\}^k$ and all $n_1<\ldots< n_k\in \N$. Then $\kappa(X,\ell+1)\geq 1/p$.
\end{cor}

\begin{proof}
This follows analogously to the proof of Theorem \ref{ThmBadDualsImpliesCoKRp} but using Corollary \ref{CorCORHighlySep} instead of Lemma \ref{LemmaCORSeqInHighdual},
\end{proof}

For the remainder of this section, we discuss some applications of our work to the class of spaces constructed by J. Lindenstrauss in \cite{Lindenstrauss1971}. Precisely, let  $p\in (1,\infty)$, let $X$ be a separable Banach space and fix a dense sequence $(x_n)_n$ in $\partial B_X$. Define $Y$ by letting
\[Y=\Big\{(\lambda_n)_n\in \R^\N\mid \|(\lambda_n)_n\|_Y\coloneqq \sup_{0=p_0<\ldots< p_k} \Big(\sum_{j=1}^k\Big\| \sum_{n=p_{j-1}+1}^{p_j}\lambda_n x_n\Big\|_X^p\Big)^{1/p}<\infty\Big\}.\]
The space $(Y,\|\cdot\|_Y)$ is a Banach space and if $(\lambda_n)_n\in Y$, then $\sum_n\lambda_nx_n$ converges. Let $Q:Y\to X$ be the bounded linear map $Q((\lambda_n)_n)=\sum_n\lambda_nx_n$ and define  \[Z_{p,X}=\ker (Q).\]
 J. Lindenstrauss showed  that $Z^{**}_{p,X}/ Z_{p,X}= X$ and $Z^{***}_{p,X}=Z^{*}_{p,X}\oplus X^*$ (see \cite{Lindenstrauss1971}, Theorem in Page 279 and Corollary 1). In particular, $Z_{p,X}$ and $Z^*_{p,X}$ are both separable. Furthermore, it was proven in Theorem 2.1 of \cite{CauseyLancien2017} that  $Z_{p,X}$ is $p$-AUSable and that $Z_{p,X}^*$ is $p'$-AUSable, where $p'$ is the conjugate of $p$, i.e.,  $1/p+1/p'=1$.\footnote{Those resutls were actually only proven for $p=p'=2$ in \cite{Lindenstrauss1971} and \cite{CauseyLancien2017}, but a simple adaptation of their prove give us this general result.} In particular, if $X$ is infinite dimensional, those spaces are not quasi-reflexive. Therefore,  it is  natural to look at those spaces when looking for counter-examples for the existence of concentration inequalities. We dedicate the rest of this section to this task.

\begin{remark}
Notice that the definition of $Z_{p,X}$ depends on the sequence $(x_n)_n$. By abuse of notation, for now on we  forget about $(x_n)_n$.  Although the space $Z_{p,X}$ may depend on the sequence, the properties the space has which interest us do not.  
\end{remark}
\begin{cor}\label{CorEmbGraphsLindenstraussSpaces}
Let $p\in (1,\infty)$, $q\in [1,\infty)$ and $q'\in (1,\infty]$, with $1/q+1/q'=1$. The following holds. \begin{enumerate}[(i)]
\item $\kappa( Z_{p,\ell_q},2)\geq 1/q$.
\item $\kappa( Z^*_{p,\ell_q},2)\geq 1/q'$.
\item  $\kappa(Z^*_{p,c_0},2)=1$.
\item $\kappa(Z_{p,c_0},3)=1$.
\end{enumerate}
Moreover, the supremums above are all attained. 
\end{cor}

\begin{proof}
(i) Since $Z^{**}_{p,\ell_q}/Z_{p,\ell_q}=\ell_q$, $Z^{**}_{p,\ell_q}/Z_{p,\ell_q}$ contains a sequence  equivalent to the standard  $\ell_q$-basis. Let $(z_n)_n$ be a bounded sequence in $Z^{**}_{p,\ell_q}$ so that $([z_n])_n$ is equivalent to the $\ell_q$-basis, where $z\in Z^{**}_{p,\ell_q}\mapsto [z]\in Z^{**}_{p,\ell_q}/Z_{p,\ell_q}$ is the standard quotient map. Then, for some $C>0$, it follows that  \[\Big\|\sum_{j=1}^kz_{n_j}\Big\|\geq Ck^{1/p},\]
 for all $k\in\N$, all $(\eps_j)_{j=1}^k\in \{-1,1\}^k$ and all $n_1<\ldots<n_k\in \N$. The result  follows immediately from Corollary \ref{CorNonColapsingEmbGraphs} applied to $(z_n)_{n\in\N}$.
 
  (ii) and (iii)  Since $Z^{***}_{p,\ell_q}=Z^*_{p,\ell_q}\oplus \ell_{q'}$ (resp. $Z^{***}_{p,c_0}=Z^*_{p,c_0}\oplus \ell_{1}$), let  $(z_n)_{n\in\N}$ be a sequence in $Z^{***}_{p,\ell_q}$ (resp. $Z^{***}_{p,c_0}$) which is equivalent to the standard  $\ell_{q'}$-basis (resp. $\ell_1$-basis) and apply Corollary  \ref{CorNonColapsingEmbGraphs}.
  
  (iv) Since $Z^{***}_{p,c_0}=Z^{*}_{p,c_0}\oplus \ell_1$, it follows that $Z^{(4)}_{p,c_0}=Z^{**}_{p,c_0}\oplus \ell_\infty$. Let  $(z_n)_{n\in\N}$ be a sequence in $Z^{(4)}_{p,\ell_q}$   equivalent to the standard  $\ell_{1}$-basis and apply Corollary  \ref{CorNonColapsingEmbGraphs}.
\end{proof}

If one wants to obtain a strengthening of Theorem \ref{TheoremLancienRajaGEN} to some class of non-quasi-reflexive spaces, a natural strategy is to look for  weakenings for the property $\dim(X^{**}/X)<\infty$, e.g., $X$ is complemented in $X^{**}$, $X^{**}/X$ is reflexive, $X^{**}/X$ is AUSable, etc. Corollary \ref{CorEmbGraphsLindenstraussSpaces} gives us counter-examples for some of those weakenings. 

\begin{cor}\label{CorCorCor}
Let $p\in (1,\infty)$. The following holds.
\begin{enumerate}[(i)]
\item \label{CorXAUSCompBiDualReflAUS} For any $q\in (1,\infty)$, there exists a $p$-AUSable  Banach space $X$ which is complemented in its bidual,  $X^{**}/X$ is reflexive and $q$-AUSable, but   $X$ has $\coKR(q)$.
\item \label{CorXAUSBiDualSuperAUS}
 There exists a $p$-AUSable  Banach space $X$ so that $X^{**}/X$ is $q$-AUSable for all $q\in (1,\infty)$, but $X$ has $\coKR(1)$.
\item \label{CorXAUSBiDualComplVerynotKRp}
There exists a $p$-AUSable  Banach space $X$ complemented in its bidual, but so that $X$ has $\coKR(1)$.
\end{enumerate}\qed
\end{cor}

Corollary \ref{CorCorCor}\eqref{CorXAUSCompBiDualReflAUS} suggests that in order to obtain a non-quasi-reflexive Banach space $X$ which has $\KR(p)$, one should at least restrict themselves to $p$-AUSable Banach spaces $X$ which are complemented in their biduals and so that $X^{**}/X$ is  $p$-AUSable. That is, simply requiring   $X^{**}/X$ to be $q$-AUSable for some $q<p$ is not enough. Unfortunately, as we see below, those requirements are still not enough (at least  if one does   require $X^{**}/X$ to be reflexive).

Fix $p\in (1,\infty)$. Given a separable Banach space $X$,  define inductively a  finite sequence of  Banach spaces with separable dual $(E_{i}(p,X))_{i=0}^\infty$ by setting  \[E_{0}(p,X)=Z_{p,X}\ \text{ and }\    E_{i+1}(p,X)= E_{0}(p,E_{i}(p,X)),\] 
for all $i\in\N$. We now list  some properties of $(E_{i}(p,X))_{i=0}^\infty$ which follow straightforwardly from the results in \cite{Lindenstrauss1971} and \cite{LancienRaja2017} mentioned above. In what follows,  $p'$ denotes the conjugate of $p$.

\begin{enumerate}[(i)]
\item $E_{i}(p,X)$ is $p$-AUSable, for all $i\in\N\cup\{0\}$.
\item\label{ItemEStarAUS} $E^*_{i}(p,X)$ is $p'$-AUSable, for all $i\in\N\cup\{0\}$.
\item $E_{i}^{**}(p,X)/E_{i}(p,X)=E_{i-1}(X,p)$, for all $i\in\N$.
\item\label{ItemLindCompSpaces} $E_{i}^{***}(p,X)=E^*_{i}(p,X)\oplus E^*_{i-1}(p,X)$, for all $i\in\N$.
\end{enumerate}
Some less trivial properties of the spaces $(E_{i}(p,X))_{i=0}^\infty$  will be needed. For this, we have the next lemma.

\begin{lemma}\label{LemmaEstarLinden}
Let $p,p'\in (1,\infty)$, with $1/p+1/p'=1$,  let $X$ be a separable Banach space and consider the family $(E_{i}(p,X))_{i=0}^\infty$ defined above. Fix $\ell\in\N$. The following holds.
\begin{enumerate}[(i)]\setcounter{enumi}{4}
\item\label{ItemEstarSeparable} $E_{\ell}^{(2\ell+2)}(p,X)$ is separable,
\item\label{ItempAUSHighDuals} $ E_{\ell}^{(2\ell+1)}(p,X) $ is $p'$-AUSable, and 
\item\label{ItemOplusSumAUSXdual} $E_{\ell}^{(2\ell+3)}(p,X)= X^*\oplus Z$, for some $p'$-AUSable Banach space $Z$.
\end{enumerate}
\end{lemma}

\begin{proof}
The proof is a simple induction on $\ell$. First of all, notice that  \eqref{ItemEstarSeparable} and \eqref{ItempAUSHighDuals} are equivalent to  
\begin{enumerate}[(i)]\setcounter{enumi}{7}
\item\label{ItemEstarSeparablePRIMA} $E_{\ell}^{(2j+2)}(p,X)$ is separable, for all $j\in \{0,\ldots, \ell\}$, and
\item\label{ItempAUSHighDualsPRIMA} $ E_{\ell}^{(2j+1)}(p,X) $ is $p'$-AUSable, for all $j\in \{0,\ldots,\ell\}$,
\end{enumerate}
respectively. We shall now prove   \eqref{ItemOplusSumAUSXdual},  \eqref{ItemEstarSeparablePRIMA} and  \eqref{ItempAUSHighDualsPRIMA} for $\ell$.

 Say $\ell=0$. Since both  $E_{0}(p,X)$ and $E^{**}_{0}(p,X)/E_{0}(p,X)=X $  are separable, so is   $E^{**}_{0}(p,X)$. As $E_{0}(p,X)=Z_{p,X}$, $E^*_{0}(p,X)$ is $p'$-AUSable and $E^{***}_{0}(p,X)=E^{*}_{0}(p,X)\oplus X^*$. So, the result follows for $\ell=0$. Say the result holds for $\ell-1$ and let us show that it holds for $\ell$. 

 As in the case $\ell=0$, $E_{\ell}^{**}(p,X)$ is clearly separable. Hence, since \eqref{ItemEStarAUS}  implies that  \eqref{ItempAUSHighDualsPRIMA}  holds for $j=0$, it follows that  \eqref{ItemEstarSeparablePRIMA} and  \eqref{ItempAUSHighDualsPRIMA} hold for $j=0$ and $\ell$. An induction within the induction now takes place. Fix $j\in\{1,\ldots,\ell\}$ and assume that  \eqref{ItemEstarSeparablePRIMA} and  \eqref{ItempAUSHighDualsPRIMA} hold for $j-1$. By  \eqref{ItemLindCompSpaces},
\[E_{\ell}^{(2j+2)}(p,X)=E^{(2(j-1)+2)}_{\ell}(p,X)\oplus E^{(2(j-1)+2)}_{\ell-1}(p,X).\]
The induction hypotheses imply that  both $E^{(2(j-1)+2)}_{\ell}(p,X)$ and $ E^{(2(j-1)+2)}_{\ell-1}(p,X)$ are separable,  so $E_{\ell}^{(2j+2)}(p,X)$ is also separable and \eqref{ItemEstarSeparablePRIMA} holds. By  \eqref{ItemLindCompSpaces},
\[E_{\ell}^{(2j+1)}(p,X)=E^{(2(j-1)+1)}_{\ell}(p,X)\oplus E^{(2(j-1)+1)}_{\ell-1}(p,X).\]
By the induction hypotheses,  the spaces $E^{(2(j-1)+1)}_{\ell}(p,X)$ and $ E^{(2(j-1)+1)}_{\ell-1}(p,X)$ are $p'$-AUSable, so $E_{\ell}^{(2j+1)}(p,X)$ is $p'$-AUSable. This finishes the second induction and it  shows that  \eqref{ItemEstarSeparablePRIMA} and  \eqref{ItempAUSHighDualsPRIMA} hold for all $j\leq \ell$. 

Using    \eqref{ItemLindCompSpaces} once again, it follows that  
\[E_{\ell}^{(2\ell+3)}(p,X)=E^{(2\ell+1)}_{\ell}(p,X)\oplus E^{(2(\ell-1)+3)}_{\ell-1}(p,X).\]
Since \eqref{ItempAUSHighDualsPRIMA} holds, $E^{(2\ell+1)}_{\ell}(p,X)$ is $p'$-AUSable. By our (first) induction hypothesis, $E^{(2(\ell-1)+3)}_{\ell-1}(p,X)=X^*\oplus Z$, for some $p'$-AUSable Banach space $Z$. Therefore,  \eqref{ItemOplusSumAUSXdual} holds.
\end{proof}

\begin{cor}\label{CorSpaceWithDualsAUSComplemBidualBUTBad}
Let $p\in (1,\infty)$ and $X$ be a dual space with separable predual. For all $\ell\in\N$, there exists a dual Banach space $E\coloneqq E(p,\ell,X)$ with the following properties
\begin{enumerate}[(i)]
\item $E^{(2\ell+1)}$ is separable,
\item $E^{(2\ell)}$ is $p$-AUSable, and 
\item  $E^{(2\ell+2)}=X\oplus Z$, for some $p$-AUSable Banach space $Z$.\end{enumerate}
In particular, for all $\ell\in\N$, there exists a Banach space $E$ which is complemented in its bidual, $E^{(2\ell)}$ is $p$-AUSable and $\kappa(E)=1$.
\end{cor}

\begin{proof}
Fix $p\in(1,\infty)$, $X$ and $\ell\in\N$. Let $p'$ be the conjugate of $p$ and let $X_*$ be a separable predual of $X$. Let $E=E^*_\ell(p',X_*)$. The result follows from Lemma \ref{LemmaEstarLinden}.

 For the last statement, simply take $E=E(p,\ell,\ell_1)$ and apply Corollary \ref{CorNonColapsingEmbGraphs}.
\end{proof}

\section{Weakly sequentially homeomorphic  Lipschitz equivalence}\label{SectionWSCLipEquiv}

The purpose of this section is to show that the concept of coarse Lipschitz embeddability by weakly sequentially continuous maps is strictly weaker than isomorphic embeddability (Theorem \ref{ThmWSCCoarseEmbWeakerIsomorphcEmb}). For that, we show that the famous example of (non-separable) Lipschitz isomorphic spaces which are not linearly isomorphic constructed in \cite{AharoniLindenstrauss1978} is also an example of non-isomorphic spaces which are weakly sequentially homeomorphically Lipschitzly equivalent.

Let $I$ be a set. Denote by $c_{00}(I)$ the set of all finitely supported maps $I\to \R$ and let $c_0(I)$ be the completion of $c_{00}(I)$ endowed with the supremum norm. 

\begin{prop}\label{PropWSC}
There exist a weakly sequentially continuous Lipschitz map $f:c_0(2^{\aleph_0})\to \ell_\infty$ and a bounded linear map $q:\ell_\infty\to c_0(2^{\aleph_0})$ so that $q\circ f=\Id_{c_0(2^{\aleph_0})}$.
\end{prop}

\begin{proof}
We  show that a certain map $f:c_0(2^{\aleph_0})\to \ell_\infty$ constructed in \cite{AharoniLindenstrauss1978} is weakly sequentially continuous. Precisely, let $I$ be an index set with the cardinality of the continuum and let $(A_i)_{i\in I}$ be a family on infinite subsets of $\N$ such that $A_i\cap A_j$ is finite, for all $i\neq j$ in $I$. Let $(e_i)_{i\in I}$ denote the standard unit basis of $c_0(I)$. Fix $x=\sum_{i\in I}x(i)e_i\in c_0(I)$, with $x(i)\geq 0$ for all $i\in I$, and let us define $f(x)\in \ell_\infty$.  Notice that $\{i\in I\mid x(i)\neq 0\}$ is countable. Therefore, there exists a sequence of distinct elements  $(i_n)_n$ in $I$ so that $x=\sum_{n=1}^\infty x(i_n)e_{i_n}$ and $x(i_{n+1})\leq x(i_n)$, for all $n\in \N$. Define $f(x)\in\ell_\infty$ by letting 

\[f(x)(k)=\left\{\begin{array}{ll}
x(i_1),&\  \text{ if }\ k\in A_{i_1},\\
x(i_n),&\  \text{ if }\ k\in A_{i_n}\setminus \cup_{m=1}^{n-1}A_{i_m},\\
0,& \ \text{ if } \ k\not\in \cup_{m=1}^{\infty}A_{i_m},
\end{array}\right. \]
for all $k\in\N$. One can easily verify that this definition is  independent of the choice of $(i_n)_n$. 

For an arbitrary $x=\sum_{i\in I}x(i)e_i\in c_0(I)$, define $I_+=\{i\in I\mid x(i)>0\}$ and $I_-=\{i\in I\mid x(i)<0\}$. Then, write  $x=x^+-x^-$, where $x^+=\sum_{i\in I_+}x(i)e_i$ and $x^-=-\sum_{i\in I_-}x(i)e_i$. We  set $f(x)=f(x^+)-f(x^-)$, for all $x\in c_0(I)$. This completes the definition of $f$. It was proved in \cite{AharoniLindenstrauss1978}, page 282, that $\Lip(f)\leq 2$.

The dual space $\ell_\infty^*$ is isometrically isomorphic to the space of all finitely additive finite signed measures $\mu$ on $\N$ with  bounded variation and so that every finite subset of $\N$ is $\mu$-null. The norm of an element $\mu\in\ell_\infty$ is its total variation, and the functional evaluation $\mu(\xi)$, for $\xi\in \ell_\infty$, is given by integration, i.e., $\mu(\xi)=\int \xi d\mu$ (see \cite{DunfordSchwartzBook1958Reprint1988}, page 296, Theorem 16). 

\begin{claim}\label{ClaimFormulaMu}
Let $\mu\in \ell_\infty^*$ and $x=\sum_{i\in I}x(i)e_{i}\in c_0(I)$. Then 
\[\mu(f(x))=\sum_{i\in I} x(i)\mu(A_i).\]
\end{claim}

\begin{proof}
Write $x^+=\sum_{n=1}^\infty x(i_n)e_{i_n}$, where $(i_n)_{n}$ is a sequence of distinct elements of $I$ such that $x(i_{n+1})\leq x(i_n)$, for all $n\in\N$. For each $n\in\N$, the set $A_{i_n}\cap (\cup_{m=1}^{n-1}A_{i_m})$ is finite. Hence, since finite subsets of $\N$ are $\mu$-null, it follows that
\[\mu\Big(f\Big(\sum_{n=1}^Nx(i_n)e_{i_n}\Big)\Big)=\sum_{n=1}^Nx(i_n)\mu\Big(A_{i_n}\setminus\cup_{m=1}^{n-1}A_{i_m}\Big)
=\sum_{n=1}^Nx(i_n)\mu(A_{i_n}),\]
for all $N\in\N$. Hence, since  $x^+=\lim_N\sum_{n=1}^Na_{i_n}e_{i_n}$,  the continuity of $f$ implies that  
\[\mu(f(x^+))=\sum_{n=1}^\infty x(i_n)\mu(A_{i_n})=\sum_{i\in I_+} x(i)\mu(A_i).\]
Similarly, we have that $\mu(f(x^-))=\sum_{i\in I_-} x(i)\mu(A_i)$, and the claim is proved.
\end{proof}

\begin{claim}
$f$ is weakly sequentially continuous.
\end{claim}

\begin{proof}
We only need to show that if  $(x_n)_n$ is  a weakly convergent sequence in $c_0(I)$ so that $x_n=x^+_n$, for all $n\in\N$, then $\wlim_n f(x_n)=f(\wlim_n x_n)$. Indeed, let $(x_n)_n$ be an arbitrary weak convergent sequence in $c_0(I)$, say $x=\wlim_n x_n$. Since weak convergence in bounded subsets of $c_0(I)$ is equivalent to pointwise convergence, it follows that $\wlim_nx^+_n=x^+$ and $\wlim_nx^-_n=x^-$. Therefore, if  $\wlim_nf(x^+_n)=f(x^+)$ and $\wlim_nf(x^-_n)=f(x^-)$, it follows that  $\wlim_nf(x_n)=f(x)$. Also, by the Jordan decomposition theorem (see \cite{DunfordSchwartzBook1958Reprint1988}, page 98, Theorem 8), every $\mu\in \ell_\infty^*$ can be written as $\mu=\mu^+-\mu^-$, where $\mu^+,\mu^-\in\ell_\infty^*$ are positive finitely additive measures. Hence, in order to verify that  $\wlim_n f(x_n)=f(\wlim_n x_n)$, we can restrict ourselves to positive finitely additive measures.

Fix a weakly convergent sequence $(x_n)_n$  in $B_{c_0(I)}$ so that $x_n=x^+_n$, for all $n\in\N$. Say $x=\wlim_n x_n$. Fix a positive finitely additive measure $\mu\in\ell_\infty$. Since $f$ is uniformly continuous, without loss of generality, assume that $\supp(x_n)\coloneqq \{i\in I\mid x_n(i)\neq 0\}$ is finite, for all $n\in\N$. Assume for a contradiction that $ (\mu(f(x_n)))_n$ does not converge to $\mu(f(x))$. Then, by going to a subsequence,  there exists $\eps>0$ so that $|\mu(f(x))-\mu(f(x_n))|>\eps$, for all $n\in\N$.

By Claim \ref{ClaimFormulaMu}, $\mu(f(x))=\sum_{i\in I}x(i)\mu(A_i)$. Hence, pick a finite subset  $F\subset I$  such that $\sum_{i\in I\setminus F}x(i)\mu(A_i)<\eps/3$. We construct an increasing sequence of natural numbers $(n_k)_k$ and a disjoint sequence $(F_k)_k$ of finite subsets of $I$ by induction on $k$ as follows.  Since $(x_n)_n$ converges to $x$ coordinatewise, pick $n_1\in\N$ such that 
\[\Big|\sum_{i\in F} x(i)\mu(A_i)-\sum_{i\in F} x_{n_1}(i)\mu(A_i)\Big|<\frac{\eps}{3}.\]
Let $F_1=\supp(x_{n_1})\setminus F$. Assume  $(n_s)_{s=1}^{k-1}$ and $(F_s)_{s=1}^{k-1}$ have been defined. Let $F'=F\cup F_1\cup\ldots F_{k-1}$ and pick $n_k>n_{k-1}$ such that 
\[\Big|\sum_{i\in F'} x(i)\mu(A_i)-\sum_{i\in F'} x_{n_k}(i)\mu(A_i)\Big|<\frac{\eps}{3}.\]
Let $F_k=\supp(x_{n_k})\setminus F'$. This finishes the definition of  $(n_k)_k$ and $(F_k)_k$.

Let $E_1=F$ and for each $k>1$ set $E_k=F\cup F_1\cup\ldots \cup F_{k-1}$. By the definitions of $(n_k)_k$ and $(F_k)_k$, it follows that $\supp(x_{n_k})\subset F_k\sqcup E_k$ and $F\subset E_k$, for all $k\in\N$. Moreover, using Claim \ref{ClaimFormulaMu},  we have that
\begin{align*}
\sum_{i\in F_k}x_{n_k}(i)\mu(A_{i})& \geq |\mu(x_{n_k})-\mu(x)|-\Big|\sum_{i\in E_k} x(i)\mu(A_i)-\sum_{i\in E_k} x_{n_k}(i)\mu(A_i)\Big|\\
& \ \ \ \ -\sum_{i\in I\setminus E_k}x(i)\mu(A_i)\\
&\geq \eps-\frac{\eps}{3}-\frac{\eps}{3}=\frac{\eps}{3},
\end{align*}
for all $k\in\N$. Notice that $\mu(\cup_{i\in E}A_i)=\sum_{i\in E}\mu(A_i)$, for all finite subsets $E\subset I$.  Therefore, since the sequence $(F_k)_k$ is disjoint and since  $x_n(i)\leq 1$, for all $n\in\N$ and all $i\in I$, this shows that  
\begin{align*}
\mu(\N)&\geq \mu\Big(\bigcup\{A_i\mid i\in \cup_{s=1}^kF_s\}\Big)\\
&=\sum_{s=1}^k\sum_{i\in F_s}\mu(A_i)\\
&\geq \sum_{s=1}^k\sum_{i\in F_s}x_{n_s}(i)\mu(A_i)\\
&\geq \frac{\eps k}{3},
\end{align*}
 for all $k\in\N$. Since $\|\mu\|<\infty$, i.e., $\mu$ has finite total  variation, this is a contradiction. 
\end{proof}

Let $q_0:\ell_\infty\to \ell_\infty/c_0$ be the quotient map. Then, it is clear that $T: e_i\in c_0(I)\mapsto q_0(\chi_{A_i})\in\ell_\infty/c_0 $  defines a linear isometry between $c_0(I)$ and $\overline{\text{span}}\{ q_0(\chi_{A_i})\mid i\in I\}$. Define $q=T^{-1}\circ q_0$. Clearly, $q\circ f=\Id_{c_0(I)}$ and this finishes the proof.
\end{proof}

\begin{proof}[Proof of Theorem \ref{ThmWSCCoarseEmbWeakerIsomorphcEmb}]
Let $g:c_0(2^{\aleph_0})\to \ell_\infty$ and $q:\ell_\infty\to c_0(2^{\aleph_0})$ be given by Proposition \ref{PropWSC}. Let $X=c_0\oplus c_0(2^{\aleph_0})$ and $Y=q^{-1}(c_0(2^{\aleph_0}))$, and define $f:X\to Y$ by letting $f(x,z)=x+g(z)$, for all $(x,z)\in X$. Then $f^{-1}(y)=(y-g\circ q(y),q(y))$, for all $y\in Y$. The map $f$ is a weakly sequentially continuous Lipschitz equivalence.  Since $c_0(2^{\aleph_0})$ does not linearly embed in $\ell_\infty$, this finishes the proof.
\end{proof}

\section{Weakly sequentially continuous coarse Lipschitz embeddings}\label{SectionWSCCoarseLipEmb}

In this section, we make use of the machinery of weakly null trees in Banach spaces and its relation with $p$-asymptotic uniform smoothness in order to study weakly sequentially continuous embeddings. The main goal of this section is to show that, in the class of Banach spaces with separable dual, AUSableness is stable under coarse Lipschitz embeddability by weakly sequentially continuous maps (see Theorem \ref{CorAUSablenessStableWSCCoarseLipEmb}).

\subsection{Weakly null tree properties and asymptotic uniform smoothness}

Asymptotic uniform smoothness is closely related to properties regarding weakly null trees in Banach spaces. In this subsection, we introduce those notions and proof the necessary results so we can obtain Theorem \ref{CorAUSablenessStableWSCCoarseLipEmb}.

\begin{defi}
Let $X$ be a Banach space, $B\subset X$ and  let $\M\subset \N$ be an infinite subset. A family $(x_{\bar{n}})_{\bar{n}\in [\M]^{\leq k}}$ in $B$ is a \emph{tree in $B$} and its \emph{height} is defined to be $k$. The tree $(x_{\bar{n}})_{\bar{n}\in [\M]^{\leq k}}$ is a \emph{weakly null tree} if  the sequence $(x_{\bar{n},n})_{n>\bar {n}}$ is  weakly null for all $\bar{n}\in [\M]^{\leq k-1}\cup\{\emptyset\}$. 
\end{defi}

The next definition is a ``tree-like'' version of the  weak $p$-Banach-Saks property.

\begin{defi}
Let $p\in (1,\infty]$ and $C>0$. A Banach space $X$ has the \emph{tree-$p$-Banach-Saks property with constant $C$} if given any $k\in\N$ and  any weakly null tree $(x_{\bar{n}})_{\bar{n}\in[\N]^{\leq k}}$ in $\partial B_X$ there exists $\bar{n}=(n_1,\ldots,n_k)\in[\N]^{k}$ so that 
\[\Big\|\sum_{j=1}^kx_{n_1,\ldots,n_j}\Big\|\leq Ck^{1/p}\]
(if $p=\infty$, we use the convention ${1/\infty}=0$). We say that $X$ has the \emph{tree-$p$-Banach-Saks property} if $X$ has the tree-$p$-Banach-Saks property with constant $C$ for some $C>0$. 
\end{defi}
 
Standard Ramsey theory implies that, for some infinite subset $\M\subset \N$, we can assume that $\|\sum_{j=1}^kx_{n_1,\ldots,n_j}\|\leq Ck^{1/p}$,
for all $\bar{n}\in[\M]^{k}$, in the definition above. Moreover,  the following holds. 

\begin{prop}\label{equivTreeBS}
Let $p\in (1,\infty]$, $C>0$ and $X$ be a Banach space. The following are equivalent.

\begin{enumerate}[(i)]
\item $X$ has the tree-$p$-Banach-Saks property with constant $C$.
\item for every  weakly null tree $(x_{\bar{n}})_{\bar{n}\in[\N]^{\leq k}}$ in $X$ there exists an infinite subset $\M\subset \N$ so that
\[\Big\|\sum_{j=1}^k\eps_jx_{n_1,\ldots,n_j}\Big\|\leq Ck^{1/p},\]
for all $\bar{n}\in[\M]^{k}$ and all $\eps_1,\ldots,\eps_k\in \{-1,1\}^{k}$.  
\end{enumerate}  
\end{prop}

\begin{proof}
(ii) implies (i) is trivial. Let us show that (i) implies (ii). Let  $(x_{\bar{n}})_{\bar{n}\in[\N]^{\leq k}}$ be a weakly null tree in $\partial B_X$. Notice that, for  any $\bar{\eps}=(\eps_1,\ldots,\eps_k)\in\{-1,1\}^k$, the tree $(z_{\bar{n}})_{\bar{n}\in[\N]^{\leq k}}$, where $z_{\bar{n}}=\eps_jx_{\bar{n}}$ for all $\bar n\in[\N]^{\leq k}$, is a normalized  weakly null tree.  Therefore, as $|\{-1,1\}^k|$ is finite,  there exists   an infinite subset $\M\subset \N$ so that 
\[\Big\|\sum_{j=1}^k\eps_jx_{n_1,\ldots,n_j}\Big\|\leq Ck^{1/p},\]
for all $\bar{n}\in [\M]^{k}$, and  all $\eps_1,\ldots,\eps_k\in \{-1,1\}$.
\end{proof}

\begin{remark}
 Clearly, if $X$ has the tree-$p$-Banach-Saks property, then $X$ has the weak $p$-Banach-Saks property. Indeed, let $(x_n)_n$ be a normalized weakly null sequence in $X$. For each $\bar{n}=(n_1,\ldots,n_j)\in[\N]^{\leq k}$ let $y_{\bar{n}}=x_{n_j}$. So, $(y_{\bar{n}})_{\bar{n}\in[\N]^{\leq k}}$ is a normalized weakly null tree. Using that $X$ has the tree-$p$-Banach-Saks property applied to the tree $(y_{\bar{n}})_{\bar{n}\in[\N]^{\leq k}}$, one gets the desired subsequence of $(x_n)_n$.
\end{remark}

\begin{defi}
Let $p\in (1,\infty]$ and $C>0$. A Banach space $X$ satisfies \emph{upper $\ell_p$-tree estimates with constant $C$} if given any $k\in\N$ and any weakly null tree $(x_{\bar{n}})_{\bar{n}\in[\N]^{\leq k}}$ in $\partial B_X$ there exists $\bar{n}\in[\N]^{k}$ so that 
\[\Big\|\sum_{j=1}^ka_jx_{n_1,\ldots,n_j}\Big\|\leq C\Big(\sum_{j=1}^k|a_j|^p\Big)^{1/p},\]
for all $a_1,\ldots,a_k\in\R$ (if $p=\infty$, we use the convention $(\sum_{j=1}^k|a_j|^p)^{1/p}=\max_j |a_j|$) . We say that $X$ satisfies \emph{upper $\ell_p$-tree estimates} if $X$ satisfies upper $\ell_p$-tree estimates with constant $C$ for some $C>0$. 
\end{defi}

Similarly to Proposition \ref{equivTreeBS}, Ramsey theory implies  that   $X$ satisfies upper $\ell_p$-tree estimates with constant $C$ if and only if  given  $k\in\N$ and a weakly null tree $(x_{\bar{n}})_{\bar{n}\in[\N]^{\leq k}}$ in $\partial B_X$ there exists an infinite subset $\M\subset \N$ so that 
\[\Big\|\sum_{j=1}^ka_jx_{n_1,\ldots,n_j}\Big\|\leq C\Big(\sum_{j=1}^k|a_j|^p\Big)^{1/p},\]
for all  $\bar{n}\in[\M]^{k}$ and all $a_1,\ldots,a_k\in\R$.

\begin{prop}\label{PropAUS}
Let $p\in (1,\infty)$ and let $X$ be a Banach space. The following holds.

\begin{enumerate}[(i)]
\item If $X$ is $p$-AUSable, then $X$ has the tree-$p$-Banach-Saks property.
\end{enumerate}
Moreover, if $X$ has separable dual,  the following holds.
\begin{enumerate}[(i)]\setcounter{enumi}{1}
\item If $X$ satisfies upper $\ell_p$-tree estimates, then $X$ is $p'$-AUSable, for all $p'\in (1,p)$.
\end{enumerate}
\end{prop}

\begin{proof}
(i) Consider  $X$ endowed with a norm making $X$ into a $p$-AUS space. 
By \cite{DimantGonzaloJaramillo2009}, Proposition 1.9, there exists $C>1$ so that 
\begin{equation}\label{EqUppermp}
\limsup_n\|x+x_n\|^p\leq \|x\|^p+C\limsup_n\|x_n\|^p,
\end{equation}
for all $x\in X$ and all weakly null sequence $(x_n)_n$ in $X$. Let us show that $X$ has the tree-$p$-Banach-Saks property with constant $\tilde{C}$, for all $\tilde{C}>C^{1/p}$. The proof follows by induction on the height of weakly null trees in $\partial B_X$. For $k=1$, the result is trivial. Assume it holds for $k-1$. Let  $(x_{\bar{n}})_{\bar{n}\in[\N]^{\leq k}}$ be a weakly null tree in $\partial B_X$, so $(x_{\bar{n}})_{\bar{n}\in[\N]^{\leq k-1}}$ is a weakly null tree in $\partial B_X$ of height $k-1$. Fix $\tilde{C}>C^{1/p}$. By the induction hypothesis and  Proposition \ref{equivTreeBS}, pick an infinite subset $\M\subset \N$ so that 
\[\Big\|\sum_{j=1}^{k-1}x_{n_1,\ldots,n_j}\Big\|\leq \tilde{C} (k-1)^{1/p},\]
for all $\bar{n}\in[\M]^{k-1}$. Fix $\bar{n}_0=(n_1,\ldots,n_{k-1})\in [\M]^{k-1}$. Letting $x=\sum_{j=1}^{k-1}x_{n_1,\ldots,n_j}$,   \eqref{EqUppermp}  implies that 
\begin{align*}
\limsup_{n_k}\Big\|\sum_{j=1}^{k-1}x_{n_1,\ldots,n_)}+x_{n_1,\ldots,n_{k-1},n_k}\Big\|^p&\leq \Big\|\sum_{j=1}^{k-1}x_{n_1,\ldots,n_j}\Big\|^p+C\\
&\leq \tilde{C}^{p}k
\end{align*}
As $\tilde{C}>C^{1/p}$ was arbitrary, the result follows. 

(ii) If $X$ satisfies upper $\ell_p$-tree estimates, then it is trivial to check that $X$ satisfies  (2) of Theorem 3 of \cite{OdellSchlumprecht2006}.\footnote{Altought this is straightforward, introducing the necessary notation would increase the size of these notes significantly. Therefore, since those notions will not be used elsewhere in this paper, we simply refer the reader to \cite{OdellSchlumprecht2006}.}  Hence, by the proof of (2)$\Rightarrow$(3) of Theorem 3 of \cite{OdellSchlumprecht2006}, if follows that $X$ is $p'$-AUSable, for all $p'\in (1,\infty)$.
\end{proof}

The next lemma follows by a simple induction on $k\in\N$, so we omit its proof. 

\begin{lemma}\label{convhull}
Let $V$ be a normed vector space and $e_1,\ldots,e_k\in V$. Then, for all $a_1,\ldots,a_k\in [-1,1]$, 
\[\sum_{j=1}^ka_je_j\in \text{conv}\Big\{\sum_{j=1}^k\varepsilon_je_j\mid \varepsilon_1,\ldots,\varepsilon_k\in\{-1,1\}\Big\}.\] In particular, if $M>0$ is such that $\|\sum_{i=1}^k\varepsilon_ie_i\|\leq M$, for all $\varepsilon_1,\ldots,\varepsilon_k\in\{-1,1\}$, then $\|\sum_{i=1}^ka_ie_i\|\leq LM$, for all $a_1,\ldots,a_k\in[-L,L]$.
\end{lemma}

\begin{lemma}\label{peps}
Let $p\in(1,\infty)$, $C>0$, and  $k\in\N$. Let  $(x_n)_{n=1}^k$ be a finite sequence in a normed space with the following property: for all $l\in\{1,\ldots,k\}$, and all $m_1<\ldots<m_l\in\{1,\ldots,k\}$, it holds that
\[\Big\|\sum_{j=1}^l\eps_jx_{m_j}\Big\|\leq C l^{1/p},\]
for all $\eps_1,\ldots, \eps_l\in \{-1,1\}$. Then, for all $\eps\in(0,p-1)$, there exists a constant $\tilde{C}=\tilde{C}(p,C,\eps)$ so  that  
\[\Big\|\sum_{n=1}^ka_nx_n\Big\|\leq \tilde{C} \Big(\sum_{n=1}^k|a_n|^{p-\eps}\Big)^{1/(p-\eps)},\]
for all $a_1,\ldots ,a_k\in \R$.
\end{lemma}

\begin{proof}
Fix $\eps\in (0,p-1)$. Let $(x_n)_{n=1}^k$ a finite sequence as above.  Let $\bar{a}=(a_1,\ldots,a_k)\in \partial B_{\ell^k_p}$. So,  $\|\bar{a}\|_{\ell^k_{p-\varepsilon}}\geq 1$ and $|a_i|\leq 1$, for all $i\in\{1,\ldots,k\}$. For each $j\in\N$, let $F_j=\{i\mid |a_i|\in (2^{-j}, 2^{1-j}]\}$. Then, Lemma \ref{convhull} implies that

\begin{align*}
\Big\|\sum_{n=1}^ka_i x_{i}\Big\|&\leq \sum_{j\in \N}\Big\|\sum_{n\in F_j}a_i x_{i}\Big\|\\
&\leq \sum_{j\in\N} C 2^{1-j}|F_j|^{1/p}\\
&\leq 2C\sum_{j\in\N}\Big(\sum_{i\in F_j}|a_i|^p\Big)^{1/p}\\
&=2C\sum_{j\in\N}\Big(\sum_{i\in F_j}|a_i|^{p-\varepsilon}\cdot|a_i|^\varepsilon \Big)^{1/p}\\
&\leq 2C\sum_{j\in\N}2^{(1-j)\varepsilon /p}\cdot \Big(\sum_{i\in F_j}|a_i|^{p-\varepsilon}\Big)^{1/p}\\
\end{align*}
Let $p'\in (1,\infty)$ be the conjugate of $p$, i.e.,  $1/p+1/p'=1$. Using H\"{o}lder's Inequality to the equation above, 
\begin{align*}
\Big\|\sum_{n=1}^ka_i x_{n}\Big\|&\leq 2C    \Big(\sum_{j\in\N} 2^{(1-j)\varepsilon p'/p}\Big)^{1/p'} \cdot \Big(\sum_{i=1}^k|a_i|^{p-\varepsilon}\Big)^{1/p}\\
&= \tilde{C}\Big(\sum_{i=1}^k|a_i|^{p-\varepsilon}\Big)^{\frac{1}{p-\varepsilon}\cdot \frac{p-\varepsilon}{p}}\\
&\leq \tilde{C}\Big(\sum_{i=1}^k|a_i|^{p-\varepsilon}\Big)^{\frac{1}{p-\varepsilon}},
\end{align*}
where \[\tilde{C}=2C    \Big(\sum_{j\in\N} 2^{(1-j)\varepsilon p'/p}\Big)^{1/p'}.\] As this holds for all $\bar{a}\in \partial B_{\ell^k_p}$, this finishes the proof.
\end{proof}

\begin{cor}\label{treepBSimpliesAUS}
Let $p\in (1,\infty]$. Let $X$ be a Banach space with the tree-$p$-Banach-Saks property. 

\begin{enumerate}[(i)]
\item If $p\in (1,\infty)$, then $X$ satisfies  upper $\ell_{p'}$-tree estimates, for all $p'\in(1,p)$. In particular, if $X$ has separable dual, then $X$ is $p'$-AUSable, for all $p'\in (1,p)$.
\item If $p=\infty$, then $X$ satisfies  upper $\ell_{\infty}$-tree estimates. 
\end{enumerate}
\end{cor}

\begin{proof}
(i) Let $C>0$ be a constant witnessing that $X$ has the tree-$p$-Banach-Saks property. Let $(x_{\bar{n}})_{\bar{n}\in[\N]^{\leq k}}$ be a weakly null sequence in $\partial B_X$. As $X$ has the tree $p$-Banach-Saks property with constant $C$, one gets that given  $m_1<\ldots < m_l\in\{1,\ldots,k\}$ and   $\M\subset \N$, there exists $\bar{n}\in[\M]^{k}$ such that
\[\Big\|\sum_{i=1}^l\varepsilon_i x_{n_{m_1},\ldots,n_{m_i}}\Big\|\leq Cl^{1/p},\]
for all $\varepsilon=(\varepsilon_i)_i\in\{-1,1\}^l$.  Since $|\{1,\ldots,k\}|^{\leq k}$ is finite, by Ramsey theory,  there exists an infinite  $\M\subset \N$ so that the inequality above holds for all $m_1<\ldots < m_l\in\{1,\ldots,k\}$  and all  $\bar{n}\in[\M]^{k}$.
 By Lemma \ref{peps}, the result follows.

The last statement follows from Proposition \ref{PropAUS}(ii).

(ii) The proof follows very similarly to the proof of Item (i). The only difference being that Lemma \ref{peps} is not needed, instead, we only need to use Lemma \ref{convhull}.
\end{proof}

\subsection{Nonlinear weakly sequentially continuous embeddings and AUSness}

We  now use the results above to obtain a stability result for weakly sequentially continuous coarse Lipschitz embeddings.

\begin{thm}\label{WCoarseLipEmbTree}
Let $p\in (1,\infty)$. Let $X$ and $Y$ be Banach spaces and assume that $Y$ satisfies upper $\ell_{p}$-tree estimates. If $X$ coarse Lipschitzly embeds into $Y$ by a map which is weakly sequentially continuous, then $X$ satisfies upper $\ell_{p'}$-tree estimates, for all $p'\in (1,p)$.
\end{thm}

\begin{proof}
Fix $C>0$ so that  $Y$ satisfies upper $\ell_p$-tree estimates with constant $C$.  Let $f:X\to Y $ be a weakly sequentially continuous coarse Lipschitz embedding. Without loss of generality, assume that $f(0)=0$. Fix $L,\eps>0$ so that 
\begin{equation}\label{cL}
L^{-1}\|x-y\|-\eps\leq \|f(x)-f(y)\|\leq L\|x-y\|+\eps,\ \ \text{for all} \ \ x,y\in X.
\end{equation}
Replacing $f$ with $x\mapsto f(nx)/n$, for $n\in\N$ large enough, we can assume that $L^{-1}-\eps>0$. 

Fix $k\in\N$ and  let $(x_{\bar{n}})_{\bar{n}\in[\N]^{\leq k}}$ be a weakly null sequence in $\partial B_X$.  Define a tree $(y_{\bar{n}})_{\bar{n}\in[\N]^{\leq k}}$ in $Y$ by setting $y_{n_1}=f(x_{n_1})$, for all $n_1\in\N$,  and \[y_{n_1,\ldots,n_j}=f(x_{n_1}+\ldots+ x_{n_1,\ldots,n_j})-f(x_{n_1}+\ldots+x_{n_1,\ldots,n_{j-1}}),\]
for all $(n_1,\ldots,n_j)\in[\N]^{\leq k}\setminus[\N]^{ 1}$.  As $(x_{\bar{n}})_{\bar{n}\in[\N]^{\leq k}}$ is weakly null and as $f$ is weakly sequentially continuous, it follows that $(y_{\bar{n}})_{\bar{n}\in[\N]^{\leq k}}$ is weakly null. By  \eqref{cL},  $\|y_{\bar{n}}\|\in(L^{-1}-\eps, L+\eps)$, for all $\bar{n}\in[\N]^{\leq k}$. Hence, $(y_{\bar{n}})_{\bar{n}\in [\N]^{\leq k}}$ is semi-normalized and it follows that the normalized tree $(y_{\bar{n}}/\|y_{\bar{n}}\|)_{\bar{n}\in [\N]^{\leq k}}$ is weakly null. By the choice of $C$, there exists $\M\subset \N$ so that 
\[\Big\|\sum_{j=1}^k y_{n_1,\ldots,n_j}\Big\|\leq (L+\eps)C k^{1/p},\]
for all $\bar{n}\in[\M]^{k}$. Since $f(x_{n_1}+\ldots+x_{n_1,\ldots,n_k})=\sum_{j=1}^ky_{n_1,\ldots,n_k}$,   \eqref{cL} implies  that 
\[\Big\|\sum_{j=1}^kx_{n_1,\ldots,n_j}\Big\|\leq L(L+\eps)Ck^{1/k}+\eps L,\]
for all $\bar{n}\in[\M]^{k}$. This shows that $X$ has the tree-$p$-Banach-Saks property. Hence, by Corollary \ref{treepBSimpliesAUS}(i), the result holds.
\end{proof}

If $p=\infty$, we obtain results to weakly sequentially continuous coarse embeddings. 

\begin{thm}\label{TheoremWCoarseEmbTree}
 Let $X$ and $Y$ be Banach spaces and assume that $Y$ satisfies upper $\ell_{\infty}$-tree estimates. If $X$ coarsely embeds into $Y$ by a map which is weakly sequentially continuous, then $X$ satisfies upper $\ell_{\infty}$-tree estimates.
\end{thm}

\begin{proof}
We proceed as in Theorem \ref{WCoarseLipEmbTree}. Precisely, fix $C>0$ so that  $Y$ satisfies upper $\ell_\infty$-tree estimates with constant $C$.  Let $f:X\to Y $ be a weakly sequentially continuous coarse  embedding with $f(0)=0$. By rescaling $f$ if necessary, assume that $\rho_f(1)>0$. For some  $L>0$, if follows that 
\begin{equation}
\rho_f(\|x-y\|)\leq \|f(x)-f(y)\|\leq L\|x-y\|+L,\ \ \text{for all} \ \ x,y\in X.
\end{equation}
Proceeding as in the proof of Theorem \ref{WCoarseLipEmbTree}, it follows that for all weakly null trees $(x_{\bar n})_{\bar n\in[\N]^{\leq k}}$ in $\partial B_X$, there exists an infinite  $\M\subset\N$ such that 
\[\rho_f\Big(\Big\|\sum_{j=1}^kx_{n_1,\ldots,n_k}\Big\|\Big)\leq 2LC,\]
for all $\bar n\in[\M]^k$. Since $f$ is coarse, this shows that $X$ has the tree-$\infty$-Banach-Saks property. By Corollary \ref{treepBSimpliesAUS}(ii), $X$ satisfies upper $\ell_{\infty}$-tree estimates.
\end{proof}

In the case where the target space is $c_0$, the proof of Theorem \ref{WCoarseLipEmbTree}  allows us to get a much stronger result.

\begin{proof}[Proof of Theorem \ref{ThmWSCCoarseIntoc0}]
A (slightly) new terminology is needed. A Banach space $X$ satisfies  \emph{infinite upper $\ell_\infty$-tree estimates} if there exists $C>0$ such that for all weakly null trees $(x_{\bar n})_{\bar n\in[\N]^{<\omega}}$ in $\partial B_X$, there exists  $\bar n\in [\N]^\omega$ such that  \[\sup_k\Big\|\sum_{i=1}^kx_{n_1,\ldots,n_i}\Big\|\leq C.\]
Clearly, $c_0$ has this property with $1+\eps$ for all $\eps>0$. 

Let $X$ be a Banach space not containing $\ell_1$ and assume that $X$ coarsely embeds into $c_0$. Proceeding similarly as in the   proof of Theorem \ref{TheoremWCoarseEmbTree}, we obtain  that $X$ must satisfy infinite upper $\ell_\infty$-tree estimates. By Theorem 4 of \cite{OdellSchlumprecht2006}, $X$ embeds isomorphically into $c_0$.
\end{proof}

\begin{proof}[Proof Theorem \ref{CorAUSablenessStableWSCCoarseLipEmb}]
By Proposition \ref{PropAUS}(i) and Corollary \ref{treepBSimpliesAUS}, if $Y$ is $p$-AUSable, then $Y$ satisfies  upper $\ell_{p'}$-tree estimates, for all $p'\in (1,p)$. Hence, by Theorem \ref{WCoarseLipEmbTree}, $X$ satisfies upper $\ell_{p'}$-tree estimates, for all $p'\in (1,p)$. As $X$ has separable dual, this gives us that $X$ is $p'$-AUSable, for all $p'\in (1,p)$ (see Proposition \ref{PropAUS}(ii)).

The last statement follows from the fact that every AUS space is $p$-AUS, for some $p\in (1.\infty)$ (see \cite{Raja2013},  Theorem 1.2).
\end{proof}

A Banach space $X$ is \emph{asymptotic-$\ell_\infty$} if there exists $C>0$ such that for all $k\in\N$, 
\[\exists X_1\in\cof(X),\ \forall x_1\in \partial B_{X_1},\ \ldots, \ \exists X_k\in\cof(X),\ \forall x_k\in \partial B_{X_k}\]
the finite sequence $(x_n)_{n=1}^k$ is $C$-equivalent to the standard basis of $\ell_\infty$.\footnote{This is often called \emph{asymptotic-$c_0$} in the literature.}

\begin{prop}\label{PropSepNotContl1}
Let $X$ be a separable Banach space not containing $\ell_1$ and satisfying upper $\ell_\infty$-estimates. Then $X$ is asymptotic-$\ell_\infty$. 
\end{prop}

\begin{proof}
 Let $(z_{j}^{(i)})_{j,i\in\N}$ be an  \emph{asymptotic model generated by normalized weakly null array}, i.e., $(z_{j}^{(i)})_{j,i\in\N}$ is a family in $\partial B_X$ such that (i) $\wlim_jz_j^{(i)}=0$, for all $i\in \N$,  (ii) there exist a sequence of positive reals $(\eps_k)_k$ converging to zero and a sequence $(e_i)_i$ in some Banach space $E$ such that for all $k\in\N$, all $(a_i)_{i=1}^k\in [-1,1]^n$ and all $k\leq n_1<\ldots < n_k\in\N$  it follows that 
\[\Big|\Big\|\sum_{i=1}^ka_iz^{(i)}_{n_i}\Big\|-\Big\|\sum_{i=1}^ka_ie_i\Big\|_E\Big|<\eps_k.\]
Fix $(e_i)_i$ and $(\eps_k)_k$ as above.

\begin{claim} The sequence $(e_i)_i$ is equivalent to the standard unit basis of $c_0$.
\end{claim}

\begin{proof}
By Proposition 4.1 and Remark 4.7.5 of \cite{FreemanOdellSarıZheng2018}, there is no loss of generality to assume that $(e_i)_i$ is unconditional. Fix $C>0$ so that $X$ satisfies upper $\ell_\infty$-estimates with constant $C$, without loss of generality, assume that $C\geq \sup_k\eps_k$. Let $k\in\N$ and for each $\bar n=(n_1,\ldots,n_i)\in [\N]^{\leq k}$, let $x_{\bar n}= z_{n_i}^{(i)}$. So, $(x_{\bar n})_{\bar n\in [\N]^{\leq k}}$ is a normalized weakly null tree. By the choice of $C$, pick $\bar n\in[\N]^k$ with $\bar n\geq k$ so that $\|\sum_{j=1}^kx_{n_1,\ldots,n_j}\|\leq C$. It follows that $\|\sum_{i=1}^ke_i\|_E\leq 2C$.  Since the constant $2C$ is independent on $k$ and since $(e_i)_i$ is unconditional, the result follows. 
\end{proof}

 By Theorem 4.6 of \cite{FreemanOdellSarıZheng2018}, this shows that $X$ is asymptotic-$\ell_\infty$.
\end{proof}

It was proved in \cite{BaudierLancienMotakisSchlumprecht2018} that if a Banach space $X$ coarsely embeds into a reflexive asymptotic-$\ell_\infty$ space, then $X$ is also reflexive and asymptotic-$\ell_\infty$.  Theorem \ref{TheoremWCoarseEmbTree} and Proposition \ref{PropSepNotContl1} gives us the following related result.  

\begin{cor}
 Let $X$ be a separable Banach space not containing $\ell_1$ and let  $Y$ be a Banach space. Assume that $X$ coarsely embeds into $Y$ by a weakly sequentially continuous map. If $Y$ is asymptotic-$\ell_\infty$,  then $X$ is asymptotic-$\ell_\infty$.\qed
\end{cor}

\section{Complexity of some asymptotic notions and applications}\label{SectionComplexity}

In this section, we mainly deal with the the difference between the  alternating $p$-Banach-Saks property and $p$-asymptotic uniform smoothness. Precisely, in the first result of this section, we  construct an example of  a space which has the alternating $p$-Banach-Saks property but does not have an equivalent  asymptotically uniformly  smooth norm. Then, we show that those classes have different complexities (Theorem \ref{pBScompcoana}). We finish this section proving  a universality  result of independent interest. More specifically, we show that if a separable Banach space contains all separable  reflexive Banach spaces coarsely, then the space must be \emph{coarsely universal} (see Theorem \ref{ThmCompBanachSaksUniversal}), i.e.,  every separable Banach space coarsely embeds into $X$.

\begin{thm}\label{pBSnottpBS}
Let $p\in (1,\infty)$. There exists a separable reflexive Banach space $X$ which has the alternating $p$-Banach-Saks property but does not have the tree $q$-Banach-Saks property, for any $q\in (1,\infty)$. In particular, $X$ is not AUSable.
\end{thm}

In order to prove Proposition \ref{pBSnottpBS}, we  first show that the weak $p$-Banach-Saks property is stable under $\ell_p$-sums. Let $(X_n,\|\cdot\|_n)$ be a sequence of Banach spaces. Define the  \emph{$\ell_p$-sum of $(X_n,\|\cdot\|_n)$} as the set of sequences $(x_n)_n$, with $x_n\in X_n$, for all $n\in\N$, so that
\[\|(x_n)_n\|\coloneqq \Big(\sum_{n=1}^\infty\|x_n\|_n^p\Big)^{1/p}<\infty.\] 
Denote this space by $(\oplus_nX_n)_{\ell_p}$. The norm $\|\cdot\|$ defined above makes $(\oplus_nX_n)_{\ell_p}$ into a Banach space.

\begin{prop}\label{pBSstablelpsum}
Let $p\in(1,\infty)$ and $C\geq 1$. Let $(X_n,\|\cdot\|_n)_n$ be a sequence of Banach spaces such that, for all $n\in\N$, $X_n$ has the weak $p$-Banach-Saks property with constant $C+\varepsilon$, for all $\varepsilon>0$. Then $(\oplus_n X_n)_{\ell_p}$ has the weak $p$-Banach-Saks property with constant $C+\varepsilon$, for all $\varepsilon>0$.
\end{prop}

Before proving Proposition \ref{pBSstablelpsum}, let us isolate a remark for future reference.

\begin{remark}\label{pBSfornonnormalizedseq}
If $X$ has the weak $p$-Banach-Saks property with constant $C>0$, and $(x_n)_n$ is a weakly null sequence such that $\lim_n\|x_n\|=a$, then, for all $\varepsilon>0$, there exists $\M\subset \N$ such that
\[\Big\|\sum_{i=1}^k x_{n_i}\Big\|\leq C (a+\varepsilon) k^{1/p},\] for all $k\in\N$ and all $k\leq n_1<\ldots<n_k\in \M$. 
\end{remark}

\begin{proof}[Proof of Proposition \ref{pBSstablelpsum}]
Let $(x_j)_j$ be a weakly null sequence in 	the unit ball of $(\oplus_nX_n)_{\ell_p}$.  For each $i\in\N$, let $P_i:(\oplus_nX_n)_{\ell_p}\to X_i$ be the natural projection, and let $(e_i)_i$ be the standard basis of $\ell_p$. For each $j\in\N$, let $z_j=\sum_i\|P_i(x_j)\|_ie_i\in  B_{\ell_p}$. As $\ell_p$ is reflexive, by taking a subsequence if necessary,  assume that $z\coloneqq \wlim_j\ z_j$ exists. Say $z=\sum_ia_ie_i$, so $z\in B_{\ell_p}$. 

Fix $k\in \N$ and $\varepsilon_0>0$.  Let $\delta\in (0,1)$  and  pick $m\in\N$ such that $\|\sum_{i>m}ka_ie_i\|_{\ell_p}\leq \delta$. Notice that $\lim_j\|P_i(x_j)\|_i=a_i$, for all $i\in\N$. Hence, passing to a subsequence, we can assume that 
\[\|z_j-z\|_{\ell_p}\leq \Big(1-\sum_{i\leq m} a_i^p+\delta\Big)^{1/p}+\delta,\]
 for all $j\in\N$.  Therefore, as  $\ell_p$ has the weak $p$-Banach-Saks property with constant $1+\varepsilon$, for all $\varepsilon>0$, by going to a further subsequence, assume that 
\[\|z_{n_1}+\ldots+z_{n_k}-kz\|_{\ell_p}\leq \Big(\Big(1-\sum_{i\leq m} a_i^p+\delta\Big)^{1/p}+2\delta\Big)k^{1/p},\]
 for all $n_1<\ldots<n_k\in\N$ (see Remark \ref{pBSfornonnormalizedseq}).  Hence, by our choice of $m$, 
\begin{align*}
\Big\|\sum_{i>m}\Big(\|P_i(x_{n_1})\|_i+\ldots+\|P_i(x_{n_k})\|_i\Big)e_i\Big\|_{\ell_p}\leq \Big(\Big(1-\sum_{i\leq m} a_i^p+\delta\Big)^{1/p}+2\delta\Big)k^{1/p}+\delta,
\end{align*}
for all $n_1<\ldots<n_k\in\N$. 

Let $\gamma>0$. As $\lim_j\|P_i(x_j)\|= a_i$, for all $i\in\N$, and as each $X_n$ has the weak $p$-Banach-Saks property with constant $C+\varepsilon$, for all $\varepsilon>0$,  pick a sequence $\M\subset \N$ such that
\[\|P_i(x_{n_1})+\ldots+P_i(x_{n_k})\|_i\leq C(a_i+\gamma)k^{1/p},\]
 for all $n_1<\ldots<n_k\in\M$, and all $i\in \{1,\ldots,m\}$ (see Remark \ref{pBSfornonnormalizedseq}).

As $C\geq 1$, we conclude that 
\begin{align*}
\|x_{n_1}+&\ldots+x_{n_k}\|^p\\
&= \sum_{i\leq m} \|P_i(x_{n_1})+\ldots+P_i(x_{n_k})\|_i^p+ \sum_{i>m} \|P_i(x_{n_1})+\ldots+P_i(x_{n_k})\|_i^p\\
&\leq C^pk\sum_{i\leq m}(a_i+\gamma)^p+ \sum_{i>m} \Big(\|P_i(x_{n_1})\|_i+\ldots+\|P_i(x_{n_k})\|_i\Big)^p\\
&\leq C^pk\sum_{i\leq m}(a_i+\gamma)^p+\Big(\Big(1-\sum_{i\leq m} a_i^p+\delta\Big)^{1/p}+2\delta\Big)k^{1/p}+\delta\Big)^p,
\end{align*}
for all $n_1<\ldots<n_k\in\N$. The proof finishes by choosing $\delta$ and $\gamma$ small enough.
\end{proof}

Before proving Theorem \ref{pBSnottpBS}, we need to introduce some terminology. This terminology is also  used in the proof of Theorem \ref{ThmCompBanachSaksUniversal}.

Denote by $\text{Tr}$ the set of all trees on $\N$, i.e., $T\in \text{Tr}$ if and only if (i) $T\subset \{\emptyset\}\cup [\N]^{<\omega}$, (ii) $\emptyset\in T$,  and (iii) $(n_1,\ldots,n_k)\in T$ implies $(n_1,\ldots,n_j)\in T$, for all $j\in\{1,\ldots,k\}$. Define a partial order $\preceq$ on $T$ by setting $(n_1,\ldots,n_j)\preceq (m_1,\ldots,m_k)$ if $j\leq k$ and $n_i=m_i$, for all $i\in\{1,\ldots,j\}$, and setting $\emptyset\preceq \bar{n}$, for all $\bar{n}\in T$. A  tree $T$ is called \emph{well-founded} if $T$ contains no strictly increasing sequence, and \emph{ill-founded} otherwise. Let $\WF$ and $\IF$ denote the set of all well-founded and ill-founded trees on $\N$, respectively. A subset  $I\subset T$ is called a \emph{segment} if it is linearly ordered with respect to $\preceq$. We say that $I_1,I_2\subset T$ are \emph{incomparable} if neither $\bar{n}\preceq\bar{m}$ nor $\bar{m}\preceq\bar{n}$, for all $\bar{n}\in I_1$, and all $\bar{m}\in I_2$. We refer to \cite{DodosBook2010}, Section 1.2, for more on trees.

Let $\cE=(e_n)_n$ be a basic sequence in a Banach space $E$. If $T\in\text{Tr}$, $x=(x(\bar{n}))_{\bar{n}\in T}\in c_{00}(T)$, and $I$ is a  segment of $T$,  write $x_{|I}=\sum_{\bar n\in I}x(\bar{n}) e_{\max (\bar n)}$, so $x_{|I}\in E$.  For each $p\in (1,\infty)$, $T\in\text{Tr}$ and $x\in c_{00}(T)$,  define
\begin{align*}
\left\|x\right\|_{p,\cE,T}=\sup\Big\{\Big(\sum_{i=1}^{n}\big\|x_{|I_i}\big\|^p_{E}\Big)^{1/p}\mid 
\ I_1, \ldots,\ I_n &\text{ incomparable segments of 
}T\Big\}.
\end{align*}
Denote  the completion of $c_{00}(T)$ under the norm $\|. \|_{p,\cE,T}$ by $X_{p,\cE,T}$. By abuse of notation, we write $X_{p,E,T}$ if the basis of $E$ is clearly specified.

\begin{proof}[Proof of Theorem \ref{pBSnottpBS}]
Fix $p\in (1,\infty)$ and let $\cE=(e_n)_n$ denote the standard basis of $\ell_1$. Given a tree $T\in \Tr$, let $X_{p,\ell_1,T}$ be the space defined above.  An easy transfinite induction on the order of $T$ (see \cite{DodosBook2010}, Section 1.2, for the definition of the\emph{ order} of a tree $T$) and Proposition \ref{pBSstablelpsum} give us that, for all well-founded trees $T\in \text{Tr}$, the space $X_{p,\ell_1,T}$ has the weak $p$-Banach-Saks property with constant $1+\varepsilon$, for all $\varepsilon>0$ (see \cite{Braga2014Czech}, Theorem 14, for a similar transfinite induction). As the $\ell_p$-sum of reflexive spaces is also reflexive, transfinite induction also gives us that $X_{p,\ell_1,T}$ is reflexive, for all well-founded $T\in \text{Tr}$. In particular,  $X_{p,\ell_1,T}$ does not contain $\ell_1$ for $T\in \WF$, and it follows that   $X_{p,\ell_1,T}$ has the alternating $p$-Banach-Saks property for all well-founded $T\in \text{Tr}$.
 
 For all $k\in\N$, let $(x_{\bar{n}})_{{\bar{n}}\in [\N]^{\leq k}}$ be the unitary weakly null tree in $ X_{p,\ell_1,[\N]^{\leq k}}$, i.e.,  $x_{\bar{n}}(\bar{m})=1$ if $\bar{n}=\bar{m}$, and $x_{\bar{n}}(\bar{m})=0$ if $\bar{n}\neq\bar{m}$. Then, for all $\bar{n}\in [\N]^{k}$, we have that
\[\Big\|\sum_{j=1}^k x_{n_1,\ldots,n_j}\Big\|=k.\] Therefore, as $(x_{\bar{n}})_{{\bar{n}}\in [\N]^{\leq k}}$ is  a weakly null tree, this gives us that, for all $q\in (1,\infty)$, the tree $q$-Banach-Saks  constant of $X_{p,\ell_1,[\N]^{\leq k}}$ is at least $k^{1-1/q}$.

Let $S$ be the Schreier tree, i.e., $ S=\{(n_1,\ldots,n_k)\in [\N]^{<\omega}\mid k\leq n_1<\ldots<n_k\}$. Then, as $S$ contains copies of $[\N]^{\leq k}$, for all $k\in\N$, it follows that  $X_{p,\ell_1,S}$ contains an isometric copy of $X_{p,\ell_1,[\N]^{\leq k}}$, for all $k\in\N$. Hence, as $\lim_k k^{1-1/q}=\infty$, for all $q\in (1,\infty)$, $X_{p,\ell_1,S}$ does not have the tree $q$-Banach-Saks property, for all $q\in (1,\infty)$. As $S$ is well-founded, this finishes the proof.
\end{proof}

The rest of this section is dedicated to prove some complexity results. Let $C[0,1]$ be the space of continuous real-valued functions on $[0,1]$ endowed with the supremum norm. Let $\SB=\{X\in C[0,1]\mid X\text{ is a closed linear subspace}\}$, and endow $\SB$ with the Effros-Borel structure (see \cite{DodosBook2010}, Chapter 2). This makes $\SB$ into  a standard Borel space and, as $C[0,1]$ contains isometric copies of every separable Banach space, $\SB$ can be seen as a coding set for the class of all separable Banach spaces. Therefore, we can talk about  Borel,  (complete) analytic and (complete) coanalytic classes of separable Banach spaces (see \cite{KechrisBook1995} for definitions).

\begin{thm}\label{pBScompcoana}
For each $p\in (1,\infty)$, the sets \[\mathrm{alt}\text{-}p\text{-}\mathrm{BS}\coloneqq \{X\in\mathrm{SB}\mid X\text{  has the alternating }p\text{-Banach-Saks property}\} \] and
\[\mathrm{Refl}+p\text{-}\mathrm{BS}\coloneqq \{X\in\SB\mid X\text{ is reflexive and has the weak }p\text{-Banach-Saks property}\}\]
are complete coanalytic. In particular, $\mathrm{alt}\text{-}p\text{-}\mathrm{BS}$ and 
\[p\text{-}\mathrm{AUSable}\coloneqq\{X\in \mathrm{SB}\mid X\text{ is }p\text{-AUSable}\}\]
have different complexities. 
\end{thm}

Before proving Theorem \ref{pBScompcoana}, we give an equivalent definition for the weak $p$-Banach-Saks property in the category of reflexive Banach spaces.

\begin{prop}\label{pBSwithoutweaklynull}
Let $p\in (1,\infty)$. A reflexive Banach space $X$ has the weak $p$-Banach-Saks property if and only if there exists $C>0$ such that, for every sequence $(x_n)_n$ in $B_X$ and every $k\in\N$, there exists $n_1<\ldots<n_{2k}\in\N$ such that
\[\Big\|\sum_{i=1}^k{x_{n_i}}-\sum_{i=k+1}^{2k}x_{n_i}\Big\|\leq C k^{1/p}.\]
\end{prop}

\begin{proof}
Let us first show the backwards direction. Let $C>0$ be as in the statement of the proposition. Let $(x_n)_n$ be a normalized weakly null sequence. Without loss of generality,  assume that $(x_n)_n$ is a basic sequence with basic constant $2$. The hypothesis and  Ramsey theory give  that, for all $k\in\N$, there exists $\M\subset \N$, such that
\[\Big\|\sum_{i=1}^k{x_{n_i}}-\sum_{i=k+1}^{2k}x_{n_i}\Big\|\leq C k^{1/p},\]
for all $n_1<\ldots<n_{2k}\in\M$. So, $\|\sum_{i=1}^kx_{n_i}\|\leq 2Ck^{1/p}$, for all $n_1<\ldots<n_k\in \M$. 

Assume $X$ has the weak $p$-Banach-Saks property with constant $C>0$. Let $(x_n)_n$ be a sequence in $B_X$. As $X$ is reflexive, by taking a subsequence,  assume that $x=\wlim_n\ x_n$ exists. Let $z_n=x_n-x$, for all $n\in\N$. So, $(z_n)_n$ is weakly null and $\|z_n\|\leq 2$, for all $n\in\M$. Fix $k\in\N$. By going to a subsequence,  assume that 
\[\Big\|\sum_{i=1}^kz_{n_i}-\sum_{i=k+1}^{2k}z_{n_i}\Big\|\leq \Big\|\sum_{i=1}^kz_{n_i}\Big\|+\Big\|\sum_{i=k+1}^{2k}z_{n_i}\Big\|\leq 4Ck^{1/p},\]
 for all $n_1<\ldots<n_{2k}\in\N$. Hence, 
\[\Big\|\sum_{i=1}^k{x_{n_i}}-\sum_{i=k+1}^{2k}x_{n_i}\Big\|=\Big\|\sum_{i=1}^kz_{n_i}-\sum_{i=k+1}^{2k}z_{n_i}\Big\|\leq 4Ck^{1/p},\]
for all $n_1<\ldots<n_{2k}\in\N$.
\end{proof}

\begin{proof}[Proof of Theorem \ref{pBScompcoana}]
Fix $p\in (1,\infty)$. The definition of the alternating $p$-Banach-Saks property given in Section \ref{SectionIntro} is clearly a coanalytic definition, so $\text{alt-}p\text{-BS}$ is coanalytic. Also, it is well known that  $\text{Refl}=\{X\in\text{SB}\mid X\text{ is reflexive}\}$ is coanalytic (see \cite{DodosBook2010}, Theorem 2.5). Hence, as the condition in the characterization of  the weak $p$-Banach-Saks property given in Proposition \ref{pBSwithoutweaklynull} is clearly coanalytic, it follows that $\text{Refl}+p\text{-BS}$ is coanalytic.

Let $\cE=(e_n)_n$ be the standard basis of $\ell_q$ for some $q\in (1,p)$. For each tree $T\in \text{Tr}$, let $X_{p,\ell_q,T} $ be the space defined in the proof of Proposition \ref{pBSnottpBS}.  Fixing an isometric copy of $X_{p,\ell_q,[\N]^{< \omega}}$ in $C[0,1]$ and identifying each $X_{p,\ell_q,T}$ with a subspace of $X_{p,\ell_q,[\N]^{<\omega}}$ in the natural way, it is easy to see that the assignment  $T\in\text{Tr}\mapsto X_{p,\ell_q,T}\in\SB$ is a Borel function. If $T$ is ill-founded, $X_{p,\ell_q,T}$ contains an isometric copy of $\ell_q$, so it has neither  the alternating $p$-Banach-Saks property nor the weak $p$-Banach-Saks property. Therefore, $T$ is well-founded if and only if $X_{p,\ell_q,T}$ has the alternating $p$-Banach-Saks property (resp. weak $p$-Banach-Saks property). Hence, the proof of Theorem  \ref{pBSnottpBS} shows  that  the well-founded trees Borel reduce to $\text{alt-}p\text{-BS}$ (resp. $\text{Refl}+p\text{BS}$), i.e., there exists a Borel map $\varphi:\Tr\to \SB$ so that $T$ is well-founded if and only if $\varphi(T)\in\text{alt-}p\text{-BS}$ (resp. $\varphi(T)\in\text{Refl}+p\text{-BS}$). This shows that $\text{alt-}p\text{-BS}$ and  $\text{Refl}+p\text{-BS}$ are complete coanalytic. 

The last statement in the theorem follows from the fact that $p$-$\mathrm{AUSable}$ is analytic (see \cite{Braga2017Studia}, page 82).
\end{proof}

\subsection{Coarsely universal Banach spaces}
N. Kalton proved that if $c_0$ coarsely embeds into a Banach space $X$, then $X^{(n)}$ is non-separable for some $n\in\N$ (\cite{Kalton2007}, Theorem 3.6). In particular $c_0$ does not coarsely embed into any reflexive Banach space. In this subsection, we show that if a separable Banach space $X$ contains coarsely every separable reflexive Banach space, then $c_0$ coarsely embeds into $X$. In particular, by N. Kalton's result, $X^{(n)}$ is non-separable for some $n\in\N$. By a famous result of I. Aharoni (\cite{Aharoni1974}, Theorem in page 288), it also follows that $X$ is coarsely universal. 
	
\begin{proof}[Proof of Theorem \ref{ThmCompBanachSaksUniversal}]
We only need to show that $C[0,1]$ coarsely embeds into $X$. Let $\A=\{Z\in \SB\mid Z\text{ coarsely embeds into } X\}$. It is easy to see that $\A$ is analytic.  Indeed, we have that
\begin{align*}
Z\in \A \ \Leftrightarrow \ &\exists (z_n)_n\in Z^\N,\ (x_n)_n\in X^\N, \\
&\Big((z_n)_n\ \text{is a net in Z}\Big)\wedge\Big( z_n\mapsto x_n \ \text{defines a coarse embedding}\Big),
\end{align*}
and it is clear that the properties ``$(z_n)_n$ is a net in $Z" $ and ``$z_n\mapsto x_n$  defines a coarse embedding$"$ are Borel.

Let $\cE=(e_n)_{n\in\N}$ be a basis for $C[0,1]$. For each $T\in \Tr$, let $X_{2,C[0,1],T}$ be the metric space define above.  An easy transfinite induction on the order of $T$ gives us that, for all well-founded trees $T\in \text{Tr}$, the space $X_{2,C[0,1],T}$ has the Banach-Saks property  (see \cite{Braga2014Czech}, Theorem 14). Also, if $T\in \IF$, it is clear that $C[0,1]$ linearly isometrically embeds into $X_{2,C[0,1],T}$. 

Let $\varphi:\Tr\to \SB$ be a Borel function so that $\varphi(T)\equiv X_T$, for all $T\in \Tr$. Then, by the discussion above, we have that
\begin{enumerate}[(i)]
\item $\varphi(T)$ has the Banach-Saks property, for all $T\in \WF$, and
\item $\varphi(T)$ contains an isometric copy of $C[0,1]$, for all $T\in \IF$.
\end{enumerate}

Suppose that $C[0,1]$ does not coarsely embed into $X$. Then $\varphi^{-1}(\A)=\WF$. Since $\WF$ is not analytic, this gives us a contradiction. Hence, $C[0,1]$ coarsely embeds into $X$.

The last statement follows from Theorem 3.6 of \cite{Kalton2007}. Indeed, since $X$ is coarsely universal, $c_0$ coarsely emebds into $X$. Therefore, by Theorem 3.6 of \cite{Kalton2007}, it follows that $X^{(n)}$ is not separable, for some $n\in\N$.  
\end{proof}

\begin{cor}
Let $X$ be separable Banach space and assume that $Y$ coarsely embeds into $ X$, for all reflexive separable Banach spaces $Y$. Then $X$ is not reflexive.\qed
\end{cor}

\section{Open Problems}\label{SectionProblems}

Many interesting problems along the lines of these notes remain open. In this section, we discuss a couple of them. Although we proved that if any of the (arbitrarily high)   iterated duals of even order ``behave badly'', then no decent concentration inequality can hold, we do not know what happens if all the  iterated duals of even order ``behave well''. Precisely, the following seems to be an interesting problem.

\begin{problem}\label{Problem1}
Let $p\in (1,\infty)$ and let $X$ be a Banach space so that $X^{(2\ell)}$ is $p$-AUSable for alll $\ell\in\N$. Is there $C>0$ so that for all $k\in\N$, all Lipschitz maps $f:([\N]^k,d_{\mathrm{H}})\to X$ and all $\eps>0$, there exists an infinite subset $\M\subset \N$ so that
\[\|f(\bar{n})-f(\bar{m})\|\leq C\Lip(f)k^{1/p}+\eps,\]
for all $(\bar{n},\bar{m})\in \cI_k(\M)$? 
\end{problem}

Assuming something extra on the quotient $X^{**}/X$ may make the problem more tangible. The following less ambitious problem seems to be a good place to start.

\begin{problem}\label{Problem2}
Let $p\in (1,\infty)$ and let $X$ be a $p$-AUS Banach space so that $X$ is complemented in $X^{**}$ and that  $X^{**}/X$ is reflexive and $p$-AUSable. Is there $C>0$ so that for all $k\in\N$, all Lipschitz maps $f:([\N]^k,d_{\mathrm{H}})\to X$ and all $\eps>0$, there exists an infinite subset $\M\subset \N$ so that
\[\|f(\bar{n})-f(\bar{m})\|\leq C\Lip(f)k^{1/p}+\eps,\]
for all $(\bar{n},\bar{m})\in \cI_k(\M)$? What if $X^{**}/X$ is $q$-AUSable, for some $q>p$?
\end{problem}

On the other hand, for those who believe that Problem \ref{Problem1} and Problem \ref{Problem2} should have a negative answer, a positive answer to the next problem would imply that. A positive answer would also give us that if $X$ coarse Lipschitzly embeds into an AUSable quasi-reflexive space, then $X$ is quasi-reflexive.

\begin{problem}
Let $X$ be a non-quasi-reflexive Banach space. Is there an $\ell\in\N$ and $L>0$ so that for all $k\in\N$, there exists $f:[\N]^{\ell k}\to X$ such that 
\[d_{\Delta}(\bar{n},\bar{m})\leq \|f(\bar{n},\bar{n}')-f(\bar{m},\bar{m}')\|\leq L\cdot d_{\mathrm{H}}\big((\bar{n},\bar{n}'),(\bar{m},\bar{m}')\big),\]
for all $(\bar{n},\bar{n}'),(\bar{m},\bar{m}')\in[\N]^{\ell k}$ with $\bar{n}, \bar{m} \in[\N]^{ k}$ and $ \bar{n}' \bar{m}'\in[\N]^{(\ell-1) k}$?
\end{problem}

\begin{problem}
Let $X$ be a Banach space which coarse Lipschitzly embeds into a quasi-reflexive AUS Banach space. Does $X$ need to be quasi-reflexive?
\end{problem}

Besides when the space $X$ has the Schur property, the only example of Banach spaces $X$ and $Y$ such that $X$ coarse Lipschitzly embeds into $Y$ by a weakly sequentially continuous map is of non-reflexive   non-separable Banach spaces. Either a positive answer of a counterexample to the next problem would be interesting.

\begin{problem}
Let $X$ and $Y$ be separable reflexive Banach spaces and assume that $X$ coarse Lipschitzly embeds into $Y$. Is there a weakly sequentially continuous coarse Lipschitz embedding from $X$ into $Y$?
\end{problem}

At last, we ask whether  a version of Theorem \ref{WCoarseLipEmbTree} holds for the weak $p$-Banach-Saks property. Precisely, we ask the following.

\begin{problem}\label{ProblemBS}
Let $X$ and $Y$ be Banach spaces. Assume that $Y$ has the weak  $p$-Banach-Saks property, for some $p\in (1,\infty)$. If $X$ coarse Lipschitzly embeds into $Y$ by a weakly sequentially continuous map, does it follow that $X$ has the weak $p$-Banach-Saks property? 
\end{problem}

\begin{acknowledgments}
Part of this article was written during my visits to Texas A\&M University, in July of 2016, and  Universit\'{e} Bourgogne Franche-Comt\'{e}, in March 2018. I am particularly thankful to Florent Baudier and  Thomas Schlumprecht for stimulating discussions which led to the results in Section \ref{SectionWSCCoarseLipEmb} and to Gilles Lancien for several suggestions which led to the results in Section \ref{SectionConcentrationInequality}, Section \ref{SectionSeparableDuals} and Section \ref{SectionComplexity}. Besides the mathematical help, I would like to express my appreciation for the warm welcoming I received in both mathematical departments mentioned above. In particular, I would like to extend my gratitude  to  Florence Lancien, Pavlos Motakis, Tony Prochazka, Matias Raja and Andrew Swift. At last, I would like to thank the Fields Institute for allowing me to make use of its facilities.
\end{acknowledgments}

\bibliographystyle{alpha}
\bibliography{bibliography}

\end{document}